\documentclass{SCAEOL}%SCAEOL for online version; SCAE for publication version; SCAES for the paper dedicated to somebody.
\usepackage{graphicx,amssymb,bbm,stmaryrd,dsfont,amsfonts,graphics,xcolor,multirow,cite,subfigure,epstopdf}
\usepackage[title]{appendix}

\numberwithin{equation}{section}
\numberwithin{table}{section}
\numberwithin{figure}{section}
\numberwithin{theorem}{section}

\newcommand{\qj}{I}
\newcommand{\xqj}{I_j}
\newcommand{\bpkijm}{[P^k(\xqj)]^m}
\newcommand{\bpkij}{[P^{k-1}(\xqj)]^{m}}
\newcommand{\cbvh}{\boldsymbol{V}_{h}}
\newcommand{\ih}{\mathcal{I}_h}

\newcommand{\thg}{\bar{\theta}}
\newcommand{\ajl}{\alpha^i_{j,\ell}}
\newcommand{\akw}{\boldsymbol{\alpha}^i_k}
\newcommand{\ajk}{\alpha^i_{j,k}}
\newcommand{\aok}{\alpha^i_{1,k}}
\newcommand{\ank}{\alpha^i_{N,k}}

\newcommand{\cajk}{\alpha^i_{j,k}}
\newcommand{\cajok}{\alpha^i_{j-1,k}}

\newcommand{\legendjl}{P_{j,\ell}}
\newcommand{\legendjk}{P_{j,k}}
\newcommand{\legendl}{P_\ell}

\newcommand{\cut}{\bs u_t}
\newcommand{\cuc}{\bs u_j}
\newcommand{\cpth}{\bs {p}^{(\theta)}}
\newcommand{\cpt}{\boldsymbol{p}^\mathrm{T}}

\newcommand{\cpf}{\boldsymbol{p}^-}
\newcommand{\cpz}{\boldsymbol{p}^+}

\newcommand{\cbrt}{\bs R^\mathrm{T}}
\newcommand{\cbbin}{\boldsymbol{B}^{-1}}
\newcommand{\cbahf}{\boldsymbol{A}^{\hf}}
\newcommand{\cbahfi}{\boldsymbol{A}^{-\hf}}

\newcommand{\trs}{\mathrm{T}}
\newcommand{\thkh}{(\theta)}
\newcommand{\thgkh}{(\bar{\theta})}

\newcommand{\mr}{\mathbb R}
\newcommand{\zizf}{z_i^{\pm}}
\newcommand{\zith}{z_i^{(\theta)}}
\newcommand{\zithg}{z_i^{(\bar{\theta})}}

\newcommand{\zo}{z_1}
\newcommand{\zi}{z_i}
\newcommand{\zm}{z_m}

\newcommand{\uo}{u_1}
\newcommand{\ui}{u_i}
\newcommand{\ut}{u_2}
\newcommand{\ao}{a_1}
\newcommand{\ai}{a_i}
\newcommand{\am}{a_m}

\newcommand{\uth}{u_3}
\newcommand{\um}{u_m}
\newcommand{\fo}{f_1}
\newcommand{\fm}{f_m}
\newcommand{\po}{p_1}
\newcommand{\xpm}{p_m}

\newcommand{\crij}{\boldsymbol{r}_{i,j+\hf}}
\newcommand{\clij}{\boldsymbol{\ell}_{i,j+\hf}}
\newcommand{\croj}{\boldsymbol{r}_{1,j+\hf}}
\newcommand{\crmj}{\boldsymbol{r}_{m,j+\hf}}
\newcommand{\cloj}{\boldsymbol{\ell}_{1,j+\hf}}
\newcommand{\clmj}{\boldsymbol{\ell}_{m,j+\hf}}
\newcommand{\cllj}{\boldsymbol{\ell}_{l,j+\hf}}
\newcommand{\crnj}{\boldsymbol{r}_{n,j+\hf}}

\newcommand{\cxi}{\boldsymbol{\xi}}
\newcommand{\cxiu}{\bs\xi_{\bs u}}
\newcommand{\cxiz}{\boldsymbol{\xi_z}}
\newcommand{\cxip}{\boldsymbol{\xi_p}}

\newcommand{\cxiut}{\bs \xi_{\bs u}^{\mathrm{T}}}
\newcommand{\cxiptr}{\cxi_{\bs p}^{\mathrm{T}}}

\newcommand{\cxiuxt}{(\cxiu)_x^\mathrm{T}}
\newcommand{\cet}{\boldsymbol{\eta}}
\newcommand{\cetu}{\bs \eta_{\bs u}}
\newcommand{\cetp}{\boldsymbol{\eta_p}}

\newcommand{\cxiupt}{({\cxiu})_t}

\newcommand{\cetupt}{({\cetu})_t}
\newcommand{\ceu}{\bs e_{\bs u}}
\newcommand{\ceut}{\bs e_{\bs u}^\mathrm{T}}

\newcommand{\cep}{\boldsymbol{e_p}}

\newcommand{\ceupt}{(\bs {e_u})_t}

\newcommand{\wh}{\boldsymbol{w}_h}
\newcommand{\wht}{\boldsymbol{w}_h^\mathrm{T}}
\newcommand{\uh}{\bs {u}_h}

\newcommand{\uhth}{\bs {u}_h^{(\theta)}}
\newcommand{\uhthg}{\bs {u}_h^{(\bar{\theta})}}
\newcommand{\uhzf}{\bs {u}_h^{\pm}}
\newcommand{\vh}{\boldsymbol{v}_h}
\newcommand{\vht}{\boldsymbol{v}_h^\mathrm{T}}

\newcommand{\ph}{\boldsymbol{p}_h}

\newcommand{\phth}{\boldsymbol{p}_h^{(\theta)}}
\newcommand{\phthg}{\boldsymbol{p}_h^{(\bar{\theta})}}

\newcommand{\pd}[1]{\partial{#1}}

\newcommand{\cupv}{\bs u'_{\bs v}}
\newcommand{\cvpu}{\bs v'_{\bs u}}
\newcommand{\cfpu}{\bs f'_{\bs u}}
\newcommand{\cfpv}{\bs f'_{\bs v}}

\newcommand{\vxlt}[1]{{#1}_{j-\hf}^\mathrm{T}}
\newcommand{\vxrt}[1]{#1_{j+\hf}^\mathrm{T}}
\newcommand{\vxl}[1]{#1_{j-\hf}}
\newcommand{\vxr}[1]{#1_{j+\hf}}

\newcommand{\lambij}{\lambda_{i,j+\hf}}

\newcommand{\intj}[1]{\int_{I_j}{#1} {\rm d}x}

\newcommand{\hjo}[2]{\mathcal H_j^{\theta}({#1}, {#2})}
\newcommand{\hjt}[2]{\mathcal H_j^{\bar{\theta}}({#1}, {#2})}
\newcommand{\hjbh}[2]{\mathcal H_j^{\wedge}({#1}, {#2})}

\newcommand{\hqjo}[2]{\mathcal H^{\theta}({#1}, {#2})}
\newcommand{\hqjt}[2]{\mathcal H^{\bar{\theta}}({#1}, {#2})}
\newcommand{\hqjbh}[2]{\mathcal H^{\wedge}({#1}, {#2})}
\newcommand{\intqj}[1]{\int_{I}{#1} {\rm d}x }

\newcommand{\cfh}{\bs {\hat{f}}}

\newcommand{\qh}{\sum\limits_{j=1}^{N}}
\newcommand{\xqh}{\sum\limits_{j=1}^{N-1}}
\newcommand{\qhi}{\sum\limits_{i=1}^{m}}
\newcommand{\qhl}{\sum\limits_{\ell=0}^{k}}
\newcommand{\jump}[1]{\llbracket {#1} \rrbracket}
\newcommand{\ave}[1]{\{\!\!\{{#1}\}\!\!\}}

\newcommand{\jumpxut}{\llbracket {\cxiu} \rrbracket^\mathrm{T}}
\newcommand{\jumpxiz}{\llbracket {\cxi_{\bs z}} \rrbracket}
\newcommand{\jumpxizt}{\llbracket {\cxi_{\bs z}} \rrbracket^\mathrm{T}}
\newcommand{\jumpxu}{\llbracket {\cxiu} \rrbracket}

\newcommand{\pave}{\ave{\bs p}}

\newcommand{\xiuave}{\ave{\cxiu}}
\newcommand{\xizave}{\ave{\cxiz}}

\newcommand{\hf}{{\frac12}}
\newcommand{\hfhk}{\hf}

\newcommand{\ho}{h^{k+1}}

\newcommand{\hot}{h^{2k+2}}
\newcommand{\norminf}[1]{\norm{#1}_\infty}
\newcommand{\normiD}[1]{\norm{#1}_{\infty,D}}

\newcommand{\norm}[1]{\Vert {#1}\Vert}
\newcommand{\norms}[1]{\Vert {#1}\Vert^2}
\newcommand{\normm}[1]{\Vert {#1}\Vert_M}
\newcommand{\normo}[1]{\Vert {#1}\Vert_1}
\newcommand{\normb}[1]{\Vert {#1}\Vert_{\Gamma_{\!h}}}
\newcommand{\dt}{ \frac{\rm d}{{\rm d}t}}

\newcommand{\ind}{\quad~\!}
\newcommand{\bh}[1]{{\color{blue} {#1}}}
\newcommand{\bphz}{\mathsf{P}_h^{\bar{\theta}}}
\newcommand{\bphf}{\mathsf{P}_h^{\theta}}
\newcommand{\pro}{\mathsf{P}}
\newcommand{\bphzt}{\mathbb{P}_h^{\bar{\theta}}}
\newcommand{\bphft}{\mathbb{P}_h^{\theta}}
\newcommand{\prot}{\mathbb{P}}
\newcommand{\bphzD}{\widetilde{\mathsf{P}}_h^{\bar{\theta}}}
\newcommand{\bphfD}{\widetilde{\mathsf{P}}_h^{\theta}}
\newcommand{\cxiuh}{\hat{\cxi}_{\bs u}}
\newcommand{\cxizh}{\hat{\cxi}_{\bs z}}
\newcommand{\cxiuhth}{\cxi^{(\theta)}_{\bs u}}
\newcommand{\cetuh}{\hat{\cet}_{\bs u}}
\newcommand{\cetuhth}{\cet^{(\theta)}_{\bs u}}

\newcommand{\ceuhth}{\bs e^{(\theta)}_{\bs u}}
\newcommand{\cetpthg}{{\cet}^{(\bar{\theta})}_{\bs p}}
\newcommand{\cuhh}{\hat{\bs u}_h}
\newcommand{\cphh}{\hat{\bs p}_h}
\newcommand{\zhi}{\hat{z}_i}

\newcommand{\bs}[1]{{\boldsymbol{#1}}}
\newcommand{\pw}{{\widetilde {\pro}}}
\newcommand{\pb}{{\bar {\pro}}}

\allowdisplaybreaks

\begin{document}

%Basic Information
\Year{2022} %
\Month{February}
\Vol{XX} %
\No{1} %
\BeginPage{1} %
\EndPage{XX} %
\AuthorMark{Hongjuan Zhang {\it et al.}}
\ReceivedDay{February 1, 2022}
%\AcceptedDay{December 22, 2022}
%\PublishedOnlineDay{; published online January 22, 2023}
%\DOI{10.1007/s11425-000-0000-0} % The author doesn't need fill in it.

% \title[short text for running head]{full title}{comments for title}
\title[The LDG method for nonlinear convection-diffusion systems]{Analysis of the local discontinuous Galerkin method with generalized fluxes for 1D nonlinear convection-diffusion systems}

% \author[]{Full name}{footnote}
% Remark:  One \author for one author

\author[1]{Hongjuan Zhang}{}
\author[1]{Boying Wu}{}
\author[1]{Xiong Meng}{Corresponding author}

\address[{\rm1}]{School of Mathematics, Harbin Institute of Technology, Harbin {\rm 150001}, Heilongjiang, China;}

\Emails{hongjuan.zhang@hit.edu.cn, mathwby@hit.edu.cn, xiongmeng@hit.edu.cn}
\maketitle

%     Abstract is required.

{\begin{center}
\parbox{14.5cm}{\begin{abstract}
In this paper, we present optimal error estimates of the local discontinuous Galerkin method with generalized numerical fluxes for one-dimensional nonlinear convection-diffusion systems.  The upwind-biased flux with adjustable numerical viscosity for the convective term is chosen based on the local characteristic decomposition, which is helpful in resolving discontinuities of degenerate parabolic equations without enforcing any limiting procedure. For the diffusive term, a pair of generalized alternating fluxes are considered. By constructing and analyzing generalized Gauss-Radau projections with respect to different convective or diffusive terms, we derive optimal error estimates for nonlinear convection-diffusion systems with the symmetrizable flux Jacobian and fully nonlinear diffusive problems. Numerical experiments including long time simulations, different boundary conditions and degenerate equations with discontinuous initial data are provided to demonstrate the sharpness of theoretical results.\vspace{-3mm}
\end{abstract}}\end{center}}

%  Keyword is required.

%  \subjclass is required.

%%%%%%%%%%%%%%%%%%%%%%%%%%%%%%%%%%%%%%%%%%%%%%%%%%%%%%%%%%%%
\renewcommand{\baselinestretch}{1.2}
\begin{center} \renewcommand{\arraystretch}{1.5}
{\begin{tabular}{lp{0.8\textwidth}} \hline \scriptsize
{\bf Citation:}\!\!\!\!&\scriptsize Hongjuan Zhang, Boying Wu, Xiong Meng.  Analysis of the local discontinuous Galerkin method with generalized fluxes for 1D nonlinear convection-diffusion systems.  \vspace{1mm}
\\
\hline
\end{tabular}}\end{center}

%%%%%%%%%%%%%%%%%%%%%%%%%%%%%%%%%%%%%%%%%%%%%%%%%%%%%%%%%%%%
%% Text of article.
%%%%%%%%%%%%%%%%%%%%%%%%%%%%%%%%%%%%%%%%%%%%%%%%%%%%%%%%%%%%
%    Section headings
\baselineskip 11pt\parindent=10.8pt  \wuhao

\section{Introduction}\label{sc1}
In this paper, we concentrate on optimal error estimates of the local discontinuous Galerkin (LDG) method with generalized numerical fluxes for one-dimensional nonlinear convection-diffusion systems of the form
\begin{subequations}\label{eqfn}
	\begin{align}
	\cut+\bs f(\bs u)_x&=(\bs A(\bs u) \bs u_x)_x, \quad  (x,t)\in \qj\times(0,T],\\
	\bs u(x,0)&=\bs u_0(x), \qquad \quad ~~\! x\in \qj,\label{eqg}
	\end{align}
\end{subequations}
where $\bs u(x,t)=(\uo(x,t),\dots,\um(x,t))^\trs:\mr\times\mr^{+}\rightarrow\mr^{m}$ is the vector-valued solution, $\bs f(\bs u)=(\fo(\bs u),\dots,$ $\fm(\bs u))^\trs:\mr^{m}\rightarrow\mr^{m}$ is a vector-valued flux function, $\bs A(\bs u)=(a_{ij}(\bs u))_{i=1,\dots,m}^{j=1,\dots,m}$ is positive semi-definite, and  $I=(0,2\pi)$. In our analysis, we mainly consider the periodic boundary conditions and \textcolor{red}{Dirichlet boundary conditions}, and the problems with mixed boundary conditions are numerically investigated.
We first show optimal error estimate when the flux Jacobian matrix $\cfpu$  is \textcolor{blue}{symmetric positive definite} in Sect. \ref{sc2}. Extensions to general nonlinear convection-diffusion systems with the symmetrizable flux Jacobian and fully nonlinear diffusive terms are given in Sect. \ref{sc3}.

The LDG method is an extension of the discontinuous Galerkin (DG) finite element methods \cite{CS1990,CS1989,CS19893}, which is suitable to solve partial differential equations involving high order spatial derivatives. It was first introduced by Cockburn and Shu \cite{CS19982} to solve convection-diffusion equations. Later, it was actively applied to solve various high order equations, such as the Schr\"{o}dinger equations \cite{xu2005local,TX2019}, the Navier-Stokes-Korteweg equations \cite{TXK2016} and the KdV type equations \cite{YS2002}. The main idea of the LDG method is to rewrite original high order equation into an equivalent first order system and then the standard DG method can be applied. We refer to papers \cite{Castillo2020,xs2007,BX2018,wang2020,Cheng2021,GX2021} for some incomplete recent development of the LDG methods.
%Supconvergence of the LDG method for linear and nonlinear convection-diffusion equations can be found in, e.g., \cite{L2021, LZMW2021}, in which the correction function technique is utilized to derive a superconvergent bound for projection errors.

There have been a wide range of numerical methods available for solving convection-diffusion systems. A fast iterative solver for convection-diffusion systems with
spectral elements was given in \cite{LE2011}. Rossi et al. \cite{RBM2012} studied the segmented waves in a reaction-diffusion-convection system. In \cite{ZG1995}, a local $L^2$-error analysis of the streamline diffusion method for nonstationary convection-diffusion systems was developed. In \cite{MN2016}, high-order DG schemes for the first-order hyperbolic advection-diffusion system were proposed. Michoski et al. \cite{MAP2017} extended the classical von Neumann analysis to support generalized nonlinear convection-reaction-diffusion equations discretized via high-order accurate DG methods. The entropy stable spatial-temporal DG method was proposed to solve compressible Navier-Stokes equations in \cite{MS2017}. In light of the sharp discontinuity transition of generalized local Lax-Friedrichs fluxes for nonlinear conservation laws in \cite{Li2020} and the accurate long time wave resolution of downwind-biased fluxes for KdV equations in \cite{Li2020-2}, it would be interesting to investigate the benefit that the generalized flux may bring for solving nonlinear convection-diffusion systems, especially for degenerate equations with discontinuous initial data.

This paper aims to derive optimal error estimates of the LDG method with generalized numerical fluxes for solving nonlinear convection-diffusion systems and to explore the advantages of adjustable numerical viscosities of such fluxes. In particular, due to the difference of the flux Jacobian matrix, we present the optimal error analysis in two different cases, namely the \textcolor{blue}{symmetric positive definite} flux Jacobian and symmetrizable flux Jacobian. For the first case, we apply the local characteristic decomposition \cite{CS19893,MRJ2018,ZS2006} and choose suitable generalized Gauss-Radau (GGR) projections for the leading projection errors for convective and diffusive terms to obtain optimal error estimates. For the second case, since the flux Jacobian matrix is symmetrizable, there are mainly two additional difficulties. One is that the standard GGR projections are no longer valid, due to the complexity of the balance of leading errors between the nonlinear convective term and the diffusive term. Inspired by the work in \cite{CMZ2017}, we construct a new projection for the auxiliary variable to compensate the error of the nonlinear convection term.
Another one is that some new test functions involving a \textcolor{blue}{symmetric positive definite} matrix pertaining to the symmetrizable theory should be chosen, allowing us to obtain a reasonable bound for the boundary term in $\Pi_2$ in Sec. \ref{sc312}. The analysis is also extended to the fully nonlinear convection-diffusion systems. To show flexibility of generalized numerical fluxes with different weights, a variety of numerical examples are provided, which exhibit smaller long time errors for smooth solutions and steeper discontinuity transitions for degenerate equations with discontinuous initial data when compared to the standard upwind fluxes.

The paper is organized as follows. In Sect. \ref{sc2}, we show the LDG scheme with generalized numerical fluxes for one-dimensional nonlinear convection-diffusion systems with \textcolor{blue}{symmetric positive definite} flux Jacobian, display the symmetrizable theory, and present the generalized skew-symmetry property of the DG spatial operator. The optimal error estimates for \textcolor{blue}{symmetric positive definite} flux Jacobian are derived, and the case of Dirichlet boundary conditions is discussed. Extensions of the analysis to the case of symmetrizable flux Jacobian and the fully nonlinear diffusive term are carried out in Sect. \ref{sc3}, in which some new modified GGR projections are constructed and analyzed. In Sect. \ref{sc4}, numerical experiments are presented to verify the theoretical results. Concluding remarks and comments on future work are given in Sect. \ref{sc5}.

\section{Error analysis of the \textcolor{blue}{symmetric positive definite} flux Jacobian}\label{sc2}
To clearly display the main idea of the analysis of the LDG method with generalized fluxes for nonlinear convection-diffusion systems, let us first consider the following system with a nonlinear convective term and linear diffusive term ($\bs A=diag(\ao,\dots,\am),\ai\geq0$ are all constants)
\begin{subequations}\label{eqhs}
\begin{align}
   \cut+\bs f(\bs u)_x&=\bs A\bs u_{xx}, \quad ~~~\! (x,t)\in \qj\times(0,T],\\
	\bs u(x,0)&=\bs u_0(x), \qquad x\in \qj.\label{eqssi}
\end{align}
\end{subequations}
The fully nonlinear case will be discussed in Sect. \ref{sc32}.
%where $\bs A=diag(\ao,\dots,\am)$ and $a_i\geq0,i=1,\dots,m$ are all contants.
\subsection{The LDG scheme}\label{sc21}
\subsubsection{Notation}
Let $\ih=\{\xqj=(\vxl{x},\vxr{x})\}_{j=1}^N$ be a partition of $I=(0,2\pi)$. The length of each element is $h_j=\vxr{x}-\vxl{x}$. The maximum element length is denoted by $h=\displaystyle\max_{1\leq j\leq N}h_j$. We assume the partition is quasi-uniform, namely, there exists a positive constant $\gamma$ such that $h_j\geq\gamma h$ for any $j$, as $h$ goes to zero. The discontinuous finite element space is defined as
\begin{equation*}
\cbvh = \{ \bs v\in [L^2(I)]^m: \bs v|_{\xqj} \in 	\bpkijm, ~j=1,\dots,N\},
\end{equation*}
where $P^k(\xqj)$ denotes the space of polynomials of degree at most $k$ on $\xqj$.

Since functions in $\cbvh$ may be discontinuous across cell interfaces, denote the jump and the average at $\vxr{x}$ by
\begin{equation*}
\vxr{\jump{\bs p}}=\vxr{\cpz}-\vxr{\cpf}, \quad\vxr{\pave}=\hf(\vxr{\cpz}+\vxr{\cpf}),
\end{equation*}
where $\bs p_{j+\hf}^\mp$ are values from the left and right elements, respectively.
Furthermore, we define the weighted average as
\begin{equation*}
\vxr{\cpth}=\theta\vxr{\cpf}+\bar{\theta}\vxr{\cpz},\quad \vxr{\bs f(\bs p)^{\thkh}}=\theta\bs f(\vxr{\cpf})+\bar{\theta}\bs f(\vxr{\cpz}),\quad \bar{\theta}=1-\theta.
\end{equation*}
As usual, we use $\normm{\cdot}$ to represent the length of a vector, or the spectral norm of a matrix. For any vector $\bs p=(\po,\cdots,\xpm)^\trs$, $\normm{\bs p}=(\sum_{i=1}^{m}p_i^2)^{1/2}$, and for any matrix $\bs C$, $\normm{\bs C}=\displaystyle\max_{\normm{\bs p}=1}\normm{\bs C\bs p}$.
Let $W^{\ell,p}(D)$ be the classical Sobolev space equipped with norm $\norm{\cdot}_{\ell,p,D}$. For any vector-valued function $\bs v = (v_1, \ldots, v_m)^\trs$, the $L^2$ norm is denoted by $\norms{\bs v}_D=\int_D{\|\bs v\|_M^2}{\rm d}x$ and the $L^\infty$ norm is denoted by
$\normiD{\bs v} =\displaystyle\max_{x\in D}\|\bs v(x)\|_M$. The subscripts $D, p, \ell$ will be omitted when $D=I, ~p=2$ and $\ell=0$. The \emph{broken} Sobolev space $W^{\ell,p}(\ih)$ and the corresponding norms can be defined in an analogous way. We use $\normb{\bs v}=\qh(\|\bs v^-_{j+\hf}\|_M^2+\|\bs v_{j-\hf}^+\|_M^2)^\hf$ to denote the $L^2$ norm at cell boundaries.

\subsubsection{The Cauchy-Schwarz inequality and inverse inequalities}
For any vector-valued functions $\bs p$, $\bs q$ and matrix-valued function $\bs C$, the following Cauchy-Schwarz inequality holds,
\begin{equation}\label{CSine}
|\cpt\bs C\bs q|\leq\normm{\bs C}\normm{\bs p}\normm{\bs q}.
\end{equation}
In what follows, we list some inverse inequalities of the finite element space \cite{BS1994}. For any $\bs p\in\cbvh$, there exists a positive constant $C$ independent of $\bs p$ and $h$, such that
\begin{align*}
&(\mathbf{\romannumeral1})~\norm{\partial_x\bs p}\leq Ch^{-1}\norm{\bs p},\\
&(\mathbf{\romannumeral2})~\normb{\bs p}\leq Ch^{-\hf}\norm{\bs p},\\
&(\mathbf{\romannumeral3})~\norminf{\bs p}\leq Ch^{-\hf}\norm{\bs p}.
\end{align*}

\subsubsection{The symmetrization technique}\label{sc213}
To deal with the symmetrizable flux Jacobian in Sect. \ref{sc3}, the following symmetrization procedure is essential. By the symmetrizable theory, e.g. \cite{Luo15}, if the flux Jacobian matrix $\cfpu$ is symmetrizable, we can find a mapping $\bs u(\bs v):\mr^m \rightarrow\mr^m$ such that the Jacobian matrix $ \cupv=(\pd\ui/\pd v_j)_{i=1,\ldots,m}^{j=1,\ldots,m}$ is \textcolor{blue}{symmetric positive definite} and the Jacobian matrix $\cfpv=\cfpu\cupv$ is symmetric. We also have the transformation $\bs v=\bs v(\bs u)$ with $\cvpu=(\pd v_i/\pd u_j)_{i=1,\ldots,m}^{j=1,\ldots,m}$ being \textcolor{blue}{symmetric positive definite}. If we let $\bs Q\equiv \bs Q(\bs u)=\cvpu$, from the symmetrizable theory, we know that $\cfpu$ has a strong relationship with an important symmetric matrix $\bs K$, i.e.,
\begin{equation*}
\bs Q^{1/2}\cfpu\bs Q^{-1/2}=\bs Q^{1/2}\cfpv\bs Q^{1/2}\triangleq \bs K.
\end{equation*}
Note that the symmetric matrix $\bs K$ has the same spectrum with $\cfpu$. Thus, $\bs K$ can be decomposed to $\bs K=\bs X^{-1}\tilde{\bs K}\bs X$ with $\tilde{\bs K}=diag\left\{{\lambda_i}\right\}_{i=1}^m$ and $\lambda_i$ $(i=1, \ldots, m)$ being the eigenvalues of $\cfpu$. By the above relationship, we can further obtain
\begin{equation}\label{qf}
\bs Q\cfpu=\bs Q^{1/2}\bs K \bs Q^{1/2}.
\end{equation}
The above identity is useful for us to obtain a definite viscosity term for the estimate of $\Pi_2$ in Sec. \ref{sc312}.

\subsubsection{The LDG scheme}
At first, we introduce an auxiliary variable $\bs p$ to rewrite \eqref{eqhs} as
\begin{equation*}
\cut+\bs f(\bs u)_x=\cbahf\bs p_x,\quad \bs p=\cbahf\bs u_x,
\end{equation*}
where $\cbahf=diag(\ao^\hf,\dots,\am^\hf)$. The semi-discrete LDG scheme of \eqref{eqhs} is: $\forall t\in(0,T]$, find $(\uh,\ph)\in\cbvh\times\cbvh$ such that
\begin{subequations}
	\begin{align}
	\intj{\vht(\uh)_t}&=\hjbh{\bs f(\uh)}{\vh}-\hjt{\cbahf\ph}{\vh},\\
	\intj{\wht\ph}&=-\hjo{\cbahf\uh}{\wh},
	\end{align}
\end{subequations}
hold for any $(\vh,\wh)\in\cbvh\times\cbvh$ and $j=1,\dots,N$, where the DG spatial operators $\mathcal H_j^\theta$ and $\mathcal H_j^\wedge$ depending on the choice of numerical fluxes (specified later) are
\begin{align*}\label{dgdiso}
\hjo{\cbahf\bs p}{\bs q}=&\intj{\bs q_x^\mathrm{T}\cbahf\bs p}-\vxrt{(\bs q^-)} \cbahf\vxr{\cpth}+\vxlt{(\bs q^+)}\cbahf\vxl{\cpth},\\
\hjbh{\bs p}{\bs q}=&\intj{\bs q_x^\mathrm{T}\bs p}-\vxrt{(\bs q^-)} \vxr{\hat{\bs p}}+\vxlt{(\bs q^+)}\vxl{\hat{\bs p}},
\end{align*}
and $\hqjo{\cbahf\bs p}{\bs q}=\qh\hjo{\cbahf\bs p}{\bs q}$, $\hqjbh{\bs p}{\bs q}=\qh\hjbh{\bs p}{\bs q}$. Using an argument similar to that in \cite{CMZ2017}, we have the following generalized skew-symmetry property of the DG spatial operator $\mathcal H_j^\theta$.
\begin{lemma}\label{lmm-dgvr-5}
Under the periodic boundary conditions, one has
\begin{equation}\label{dgopp}
\hqjo{\cbahf\bs p}{\bs q}+\hqjt{\cbahf\bs q}{\bs p}=0.
\end{equation}	
\end{lemma}

Following \cite{MRJ2018}, we consider the Jacobian matrix $\cfpu(\vxr{\bs u})$ in the definition of the upwind-biased flux $\cfh$. The corresponding eigenvalues, left and right eigenvectors are denoted by $\textcolor{blue}{\lambij,\clij,\crij}$ $(i=1,\dots,m)$, normalized so that $\textcolor{blue}{\cllj \crnj=\delta_{l,n}}$, and $\textcolor{blue}{\vxr{\bs R}=[\croj,\dots,\crmj],}$ $\textcolor{blue}{\vxr{\bs L}=[\cloj,\dots,}$ $\textcolor{blue}{\clmj]^\trs.}$ The upwind-biased flux is defined by the following procedure.

\begin{enumerate}
  \item Transform $\bs f(\uhzf)$ to the eigenspace of  $\cfpu(\vxr{\bs u})$,
\begin{equation*}
\zizf=\textcolor{blue}{\clij}\bs f(\bs u_h^{\pm}), \quad i=1,\dots,m.
\end{equation*}
  \item Apply the scalar upwind-biased setting to $\zizf$ in the $i$th characteristic field $(i=1,\dots,m)$, and the numerical flux $\zhi$ depends on the sign of $\textcolor{blue}{\lambij}$,
\begin{equation*}
\zhi=
\begin{cases}
\zith&\text{if $\textcolor{blue}{\lambij}\geq 0$},\\
\zithg&\text{if $\textcolor{blue}{\lambij} <  0$}.
\end{cases}
\end{equation*}
  \item Transform back to the physical field to get $\vxr{\cfh}$,
\begin{equation}\label{flx-s}
\vxr{\cfh}=\qhi\zhi\textcolor{blue}{\crij.}
\end{equation}
\end{enumerate}
In particular, when the flux Jacobian $\cfpu$ is \textcolor{blue}{symmetric positive definite},   eigenvalues of $\cfpu(\vxr{\bs u})$ are all positive. It follows from the above procedure that
\begin{subequations}
\begin{equation}\label{fluxc}
\vxr{\cfh}=\vxr{\bs f(\uh)}^{(\theta)}.
\end{equation}
For diffusive terms, the following generalized alternating numerical fluxes are chosen
\begin{equation}\label{flux1}
	\cuhh=\uhth,\quad \cphh=\phthg,
\end{equation}
or
\begin{equation}\label{flux2}
	\cuhh=\uhthg,\quad \cphh=\phth
\end{equation}
\end{subequations}
where $\theta > \frac12$ and we have omitted the subscript $j+\hf$.
The numerical initial condition is chosen as $\uh(0)=\bs \pi_h \bs u_0$, where $\bs \pi_h = (\pi_h, \ldots, \pi_h)^\trs$ is the standard  $L^2$ projection in the vector form.

\subsection{Optimal error estimates of the \textcolor{blue}{symmetric positive definite} flux Jacobian}\label{sc22}
\subsubsection{Preliminaries}
To define projection for systems, let us first recall scalar GGR projections \cite{Liu15,CMZ2017}. For $z\in H^1(\ih)$ and an arbitrary cell $\xqj$, projections \textcolor{blue}{$\bphft$} and \textcolor{blue}{$\bphzt$} are respectively defined as for any $\textcolor{blue}{j=1,\dots,N}$
\begin{subequations}\label{pro-s}
\begin{align}
\intj{(\textcolor{blue}{\bphft z})v}&=\intj{zv},\quad \forall v\in P^{k-1}(\xqj),\quad\vxr{(\textcolor{blue}{\bphft z})}^{\thkh}=\vxr{z}^{\thkh},\label{pro-sa}\\
\intj{(\textcolor{blue}{\bphzt z})v}&=\intj{zv},\quad \forall v\in P^{k-1}(\xqj),\quad\vxl{(\textcolor{blue}{\bphzt z})}^{\thgkh}=\vxl{z}^{\thgkh}. \label{pro-s2}
\end{align}
\end{subequations}

Next, consider the local linearization of the Jacobian matrix \textcolor{red}{$\cfpu(\vxr{\bs u})$ again.}
Its eigenvalues, left and right eigenvectors are \textcolor{red}{$\lambij,$ $\clij,$ $\crij$ $(i=1,\dots,m),$ and $\cllj \crnj=\delta_{l,n}$ respectively. Besides, }\textcolor{red}{$\bs R_{j+\hf}=[\croj,\dots,\crmj]$, $\bs \Lambda_{j+\hf}=diag(\lambij)_{i=1}^m$, and $\bs R_{j+\hf}^{-1}=\bs L_{j+\hf}=[\cloj,\dots,\clmj]^\trs$. Clearly, $\cfpu(\vxr{\bs u})\bs R_{j+\hf}=\bs R_{j+\hf}\bs \Lambda_{j+\hf}$.} Then, the projection of a vector $\bs u$, denoted by $\pro\bs u$ in $\xqj$, is the unique function in $\bpkijm$
determined by the following procedure.

\begin{enumerate}
  \item Transform $\bs u$ to the eigenspace of \textcolor{red}{$\cfpu(\vxr{\bs u})$},
\begin{equation*}
\zi=\textcolor{red}{\clij}\bs u, \quad i=1,\dots,m.
\end{equation*}
  \item Apply the scalar GGR projection \eqref{pro-s} to $\zi$ for the $i$th characteristic variable $(i=1,\dots,m)$, and the projection \textcolor{blue}{$\prot\zi$} depends on the sign of \textcolor{red}{$\lambij$}, i.e.,
\begin{equation*}
\textcolor{blue}{\prot\zi}=
\begin{cases}
\textcolor{blue}{\bphft\zi}&\text{if $\textcolor{red}{\lambij}\geq 0$},\\
\textcolor{blue}{\bphzt\zi}&\text{if $\textcolor{red}{\lambij}< 0$}.
\end{cases}
\end{equation*}
  \item Transform back to the physical field to get $\pro\bs u$,
\begin{equation*}
\pro\bs u=\qhi \textcolor{blue}{\prot\zi}\textcolor{red}{\crij}.
\end{equation*}
\end{enumerate}
According to the above definition, we have $\bs u=\textcolor{red}{\bs R_r}\bs z$ and $\pro\bs u=\textcolor{red}{\bs R_r} \prot \bs z$ with \textcolor{red}{$\bs R_r=\bs R_{j+\hf}$}, where $\bs z=(\zo,\dots,\zm)^\trs$ is the characteristic variable. When the matrix $\cfpu$ is \textcolor{blue}{symmetric positive definite}, we have \bh{$\prot \bs z =(\bphft\zo,\dots,\bphft\zm)^\trs$} and $\pro\bs u \triangleq \bphf\bs u$. Note that \textcolor{red}{$\bs R_r=\bs R_{j+\hf}$} is a constant matrix in each element $\xqj$, by the definition and approximation property of scalar GGR projection \textcolor{blue}{$\bphft$} in \eqref{pro-sa}, we conclude that
\begin{subequations}\label{op-pro}
\begin{equation}
(\bs u-\pro\bs u,\vh)_j=0,\quad \forall\vh\in\bpkij,\quad \vxr{(\bs u-\pro\bs u)}^{\thkh}=\bs 0,
\end{equation}
and
\begin{equation}\label{pru}
\norm{\bs u-\pro\bs u}+h\norm{(\bs u-\pro\bs u)_x}+h^\hf\normb{\bs u-\pro\bs u}\leq C\ho.
\end{equation}
We choose the projection $\bphz\bs p=(\textcolor{blue}{\bphzt \po},\dots,\textcolor{blue}{\bphzt \xpm})^\trs$ for the auxiliary variable $\bs p$. Analogously,
\begin{equation}
(\bs p-\bphz\bs p,\vh)_j=0,\quad \forall\vh\in\bpkij,\quad \vxl{(\bs p-\bphz\bs p)}^{\thgkh}=\bs 0,
\end{equation}
and
\begin{equation}\label{prp}
\norm{\bs p-\bphz\bs p}+h\norm{(\bs p-\bphz\bs p)_x}+h^\hf\normb{\bs p-\bphz\bs p} \leq C\ho.
\end{equation}
\end{subequations}

By Galerkin orthogonality, for all $(\vh,\wh) \in \cbvh\times\cbvh$, we have the error equations
%\begin{subequations}\label{erreq1}
	\begin{align*}
	\intqj{\vht\ceupt}&=\hqjo{\bs f(\bs u)-\bs f(\uh)}{\vh}-\hqjt{\cbahf\cep}{\vh},\\
	\intqj{\wht\cep}&=-\hqjo{\cbahf\ceu}{\wh},
	\end{align*}
%\end{subequations}
which, by the usual decomposition that $\ceu=(\bs u-\bphf\bs u) + (\bphf\bs u-\uh) \triangleq \cetu + \cxiu$ and $\cep=(\bs p-\bphz\bs p) + (\bphz\bs p-\ph) \triangleq \cetp + \cxip$, are
\begin{subequations}\label{erreq2}
	\begin{align}
	\intqj{\vht\cxiupt}&=-\intqj{\vht\cetupt}+\hqjo{\bs f(\bs u)-\bs f(\uh)}{\vh}-\hqjt{\cbahf\cxip}{\vh}\notag\\
	&\ind-\hqjt{\cbahf\cetp}{\vh},\\
	\intqj{\wht\cxip}&=-\intqj{\wht\cetp}-\hqjo{\cbahf\cxiu}{\wh}-\hqjo{\cbahf\cetu}{\wh}.
	\end{align}
\end{subequations}
To deal with the nonlinearity of nonlinear convective and diffusive terms, we make an \emph{a priori} assumption that
\begin{subequations}
\begin{equation}\label{pr1}
\norm{\cxiu} \le h^{\frac{3}{2}}, ~k\geq 1.
\end{equation}
 By the inverse property, one has
\begin{equation}\label{aprioroass}
\norminf{\cxiu}\leq Ch, \quad\norminf{\ceu}\leq Ch.
\end{equation}
\end{subequations}
Remark that the \emph{a priori} assumption can be verified using the technique in \cite{ZS2006}, and details are omitted.

\textcolor{red}{\begin{remark}\label{rem-defpro-1}
It is worth noting that the above GGR projection of $\bs u$ is defined based on the characteristic decomposition procedure; see, e.g., \cite[Sect. 5.6]{Luo15} in which the local Gauss-Radau projection is constructed for the purely upwind flux. Alternatively, since the Jacobian matrix is symmetric positive definite, one can also use the GGR projection for componentwise, and the optimal error estimates still hold. 	
\end{remark}}

\subsubsection{The optimal error estimate}
\begin{theorem}\label{thm-oeep-3}
Assume that the exact solution $\bs u$ of \eqref{eqhs} is sufficiently smooth, e.g., $\norm{\bs u}_{k+2}(t)$ and $\norm{\cut}_{k+1}(t)$ are bounded uniformly for any $t\in (0,T]$. Let $\uh$ be the LDG solution with the numerical fluxes \eqref{fluxc}, \eqref{flux1}. For a quasi-uniform mesh and $k \ge 1$, we have, for any $T>0$, the following error estimates
\begin{equation}
	\norm{\bs u(T)-\uh(T)} \leq C\ho,
\end{equation}
where the positive constant $C$ depends on $\norm{\bs u}_{k+2}$, $\norm{\cut}_{k+1}$, $\bs f$, but is independent of $h$.
\end{theorem}
\begin{proof}
Taking $\vh=\cxiu,\wh=\cxip$ in \eqref{erreq2}, we obtain the following error equations
%\begin{subequations}\label{erreq3}
\begin{align*}
\intqj{\cxiut\cxiupt}&=-\intqj{\cxiut\cetupt}+\hqjo{\bs f(\bs u)-\bs f(\uh)}{\cxiu}-\hqjt{\cbahf\cxip}{\cxiu}\notag\\
&\ind-\hqjt{\cbahf\cetp}{\cxiu},\\ \intqj{\cxiptr\cxip}&=-\intqj{\cxiptr\cetp}-\hqjo{\cbahf\cxiu}{\cxip}-\hqjo{\cbahf\cetu}{\cxip}.
\end{align*}
%\end{subequations}
Summing up the above equalities and using the second order Taylor expansion $\bs f(\bs u)-\bs f(\uh) = \cfpu(\bs u)\cxiu + \cfpu(\bs u)\cetu - \bs e_{\bs u}^\trs \bs H \bs e_{\bs u}$ with $\bs e_{\bs u}^\trs \bs H \bs e_{\bs u} = (\bs e_{\bs u}^\trs \bs H_1 \bs e_{\bs u}, \ldots, \bs e_{\bs u}^\trs \bs H_m \bs e_{\bs u})^\trs$ and $\bs H_i$ being the Hessian matrix \cite{MRJ2018}, we arrive at
\begin{align}\label{erreq4} \hf\dt\norms{\cxiu}+\norms{\cxip}
=\Lambda_1 +  \Lambda_2 + \Lambda_3 + \Lambda_4 +\Lambda_5,
\end{align}
where
\begin{align*}
	\Lambda_1 & =-\intqj{\cxiut\cetupt}-\intqj{\cxiptr\cetp},\\
	\Lambda_2 & =\hqjo{\cfpu(\bs u)\cxiu}{\cxiu},\\
	\Lambda_3 & =\hqjo{\cfpu(\bs u)\cetu}{\cxiu},\\
	\Lambda_4 &=-\hqjo{\bs e_{\bs u}^\trs \bs H \bs e_{\bs u}}{\cxiu},\\
	%=\hqjo{\ceut\bs H\ceu}{\cxiu},\\
	\Lambda_5 & =-\hqjt{\cbahf\cxip}{\cxiu}-\hqjo{\cbahf\cxiu}{\cxip}
-\hqjt{\cbahf\cetp}{\cxiu}-\hqjo{\cbahf\cetu}{\cxip}
\end{align*}
will be estimated separately.

By the Cauchy-Schwarz inequality, the optimal projection error estimates \eqref{op-pro} and Young's inequality, we get
\begin{subequations}\label{eres}
\begin{align}
\Lambda_1 &\leq C\ho\norm{\cxiu}+C\ho\norm{\cxip} \notag\\
&\leq C\norms{\cxiu}+\hf \norms{\cxip} + C\hot.\label{eres1}
\end{align}
Using integration by parts and taking into account the \textcolor{blue}{symmetric positive definite} matrix $\cfpu(\bs u)$, one has
\begin{align}\label{eres2}\notag
\Lambda_2 & =\qh\intj{\cxiuxt\cfpu(\bs u)\cxiu}+\qh\vxr{\jumpxut}\vxr{\big(\cfpu(\bs u)\cxiuhth\big)}\\\notag
		&=-\hf\qh\intj{\cxiut\partial_x(\cfpu(\bs u))\cxiu}-\qh\vxr{\jumpxut}\vxr{\big(\cfpu(\bs u)\xiuave\big)}\\\notag
		&\ind +\qh\vxr{\jumpxut}\vxr{\big(\cfpu(\bs u)\cxiuhth\big)}\\\notag
		&=-\hf\qh\intj{\cxiut\partial_x(\cfpu(\bs u))\cxiu}-\Big(\theta-\hf\Big)\qh\vxr{\jumpxut}\vxr{\big(\cfpu(\bs u)\jumpxu\big)}\\
		&\leq C \norms{\cxiu},
\end{align}
since $\theta > \hf.$ $\Lambda_3$ can be bounded by the local linearization of $\cfpu(\bs u)$ at $\bs u_j$ and projection property \eqref{op-pro}; it reads,
\begin{align}\label{eres3}\notag
\Lambda_3  & = \qh\intj{\cxiuxt\cfpu(\bs u)\cetu}+\qh\vxr{\jumpxut}\vxr{\big(\cfpu(\bs u)\cetuhth\big)}\\\notag
		& = \qh\intj{\cxiuxt\cfpu(\cuc)\cetu}+\qh\intj{\cxiuxt\big(\cfpu(\bs u)-\cfpu(\cuc)\big)\cetu}\\
		&\leq C\ho\norm{\cxiu} \notag\\
  &  \le  C(\norms{\cxiu} + \hot),
\end{align}
where we have also used the inverse inequality and $\normm{\cfpu(\bs u)-\cfpu(\cuc)} \le C \normm {\bs u- \cuc}\le C h$.
For the high order term $\Lambda_4$, it is easy to show that
\begin{align}
\Lambda_4 &\le C \norminf{\ceu}\big(\norm {\bs e_{\bs u}} \norm {(\cxiu)_x} +\normb {\bs e_{\bs u}} \normb \cxiu \big) \notag\\
&\le C h^{-1} \norminf{\ceu}\big(\norm \cxiu +\norm \cetu + h^\hf \normb \cetu\big)\norm \cxiu \notag\\
&\le C (\norms{\cxiu} + \hot), \label{eres4}
\end{align}
in which the bound \eqref{aprioroass} is also used.
For $\Lambda_5$, by Lemma \ref{lmm-dgvr-5} and the projection property in \eqref{op-pro}, one has
\begin{equation}\label{eres5}
\Lambda_5=0.
\end{equation}
\end{subequations}
Collecting \eqref{eres1}--\eqref{eres5} into \eqref{erreq4}, we have
\begin{equation*}
\hf\dt\norms{\cxiu}+\hf\norms{\cxip}\le C\norms{\cxiu} +  C\hot.
\end{equation*}
This, in combination with the Gronwall's  inequality, leads to the desired optimal error estimates of Theorem \ref{thm-oeep-3}.
\end{proof}

\textcolor{red}{\subsubsection{The case of Dirichlet boundary conditions}
For Dirichlet boundary conditions,
\begin{equation}\label{dirichlet}
\bs u(0,t)=\bs g_1(t), \quad \bs u(2\pi,t)=\bs g_2(t),
\end{equation}
the numerical fluxes are chosen as
\begin{equation*}
\vxr{\big(\cfh(\uh),\cuhh,\cphh\big)}=
\begin{cases}
\big(\bs f(\bs g_1),\bs g_1,\ph^+\big)_\hf, &j=0,\\
\big(\bs f(\uh)^{(\theta)},\uhth,\phthg\big)_{j+\hf}, &j=1,\dots,N-1,\\
\big(\bs f(\uh^-),\bs g_2,\ph^-\big)_{N+\hf}, &j=N.
\end{cases}
\end{equation*}
The global projections $\bphf \bs u$ and $\bphz \bs p$ are modified to $\bphfD \bs u$ and $\bphzD \bs p$ satisfying
\begin{align*}
&(\bphfD\bs u,\vh)_j=(\bs u,\vh)_j, \quad \forall\vh\in\bpkij,\\
&\vxr{(\bphfD\bs u)}^{\thkh}=\vxr{\bs u}^{\thkh}, \quad\quad\;\; \forall j\in \mathbb{Z}_{N-1}^+,\\
&f'(\bs u_{N+\hf})(\bphfD\bs u)_{N+\hf}^-=f'(\bs u_{N+\hf})\bs u_{N+\hf}^-+\cbahf(\bphzD \bs p-\bs p)_{N+\hf}^-,
\end{align*}
and 
\begin{align*}
&(\bphzD\bs p,\vh)_j=(\bs p,\vh)_j, \quad\forall\vh\in\bpkij,\\
&\vxl{(\bphzD\bs p)}^{\thgkh}=\vxl{\bs p}^{\thgkh}, \quad\quad\;\;\forall j\in \mathbb{Z}_N^+\backslash \left\{1\right\},\\
&(\bphzD\bs p)_{\hf}^+=\bs p_{\hf}^+.
\end{align*}}

\textcolor{red}{
By the above numerical fluxes and projections, we can prove the optimal error estimates for the Dirichlet boundary conditions.}

\textcolor{red}{\begin{theorem}\label{thm-oeeD-44}
Under Dirichlet boundary conditions \eqref{dirichlet}, for a quasi-uniform mesh and $k \ge 1$, we have, for any $T>0$, the following error estimates
\begin{equation*}
\norm{\bs u(T)-\uh(T)} \leq C\ho,
\end{equation*}
where the positive constant $C$ depends on $\norm{\bs u}_{k+2}$, $\norm{\cut}_{k+1}$, $\bs f$, but is independent of $h$.
\end{theorem}
\begin{proof}
Using an argument similar to that in deriving \eqref{erreq4}, we have
\begin{align*}
 \hf\dt\norms{\cxiu}+\norms{\cxip}
=\widetilde{\Lambda}_1 +  \widetilde{\Lambda}_2 + \widetilde{\Lambda}_3 + \widetilde{\Lambda}_4 +\widetilde{\Lambda}_5,
\end{align*}
where
\begin{align*}
\widetilde{\Lambda}_1 & =-\intqj{\cxiut\cetupt}-\intqj{\cxiptr\cetp},\\
\widetilde{\Lambda}_2 & =\qh\intj{\cxiuxt\cfpu(\bs u)\cxiu}+\xqh\vxr{\jumpxut}\vxr{\big(\cfpu(\bs u)\cxiuhth\big)}-\big((\cxiu^-)^\trs\bs f'(\bs u)\cxiu^-\big)_{N+\hf},\\
\widetilde{\Lambda}_3 & =\qh\intj{\cxiuxt\cfpu(\bs u)\cetu}+\xqh\vxr{\jumpxut}\vxr{\big(\cfpu(\bs u)\cetuhth\big)}-\big((\cxiu^-)^\trs\bs f'(\bs u)\cetu^-\big)_{N+\hf},\\
\widetilde{\Lambda}_4 &=-\qh\intj{\cxiuxt\ceut \bs H \ceu}-\xqh\vxr{\jumpxut}\vxr{\big((\ceuhth)^\trs\bs H \ceuhth\big)}+\big((\cxiu^-)^\trs(\ceu^-)^\trs\bs H\ceu^-\big)_{N+\hf},\\
\widetilde{\Lambda}_5 & =-\qh\bigg(\intj{\cxiuxt\cbahf\cxip}+\intj{(\cxip)_x^\trs\cbahf\cxiu}+\intj{\cxiuxt\cbahf\cetp}+\intj{(\cxip)_x^\trs\cbahf\cetu}\bigg)\\
&\ind -\xqh\vxr{\bigg(\jumpxut\cbahf\cxip^{(\bar{\theta})}
 +\jump{\cxip}^\trs\cbahf\cxiuhth+\jumpxut\cbahf\cetp^{(\bar{\theta})}+\jump{\cxip}^\trs\cbahf\cetuhth\bigg)}\\
&\ind +\big((\cxiu^-)^\trs\cbahf(\cxip^-+\cetp^-)\big)_{N+\hf}-\big((\cxiu^+)^\trs\cbahf(\cxip^++\cetp^+)\big)_{\hf}
\end{align*}
will be estimated separately.
Paralleling to the estimates of $\Lambda_1$--$\Lambda_5$ in \eqref{eres}, it is easy to show
\begin{align*}
\widetilde{\Lambda}_1 &\le C\norms{\cxiu}+\hf \norms{\cxip} + C\hot,\\
\tilde{\Lambda}_2&\le C \norms{\cxiu}-\hf\big((\cxiu^-)^\trs\bs f'(\bs u)\cxiu^-\big)_{N+\hf}-\hf\big((\cxiu^+)^\trs\bs f'(\bs u)\cxiu^+\big)_{\hf},\\
\widetilde{\Lambda}_3&\le C(\norms{\cxiu} + \hot)-\big((\cxiu^-)^\trs\bs f'(\bs u)\cetu^-\big)_{N+\hf},\\
\widetilde{\Lambda}_4&\le C (\norms{\cxiu} + \hot),\\
\widetilde{\Lambda}_5&=\big((\cxiu^-)^\trs\cbahf(\cetp^-)\big)_{N+\hf}-\big((\cxiu^+)^\trs\cbahf(\cetp^+)\big)_{\hf}.
\end{align*}
 Furthermore, with the help of the newly defined projection  $\bphfD \bs u$ and $\bphzD \bs p$ in combination with  the symmetric positive property of flux Jacobian, we have
\begin{equation*}
\hf\dt\norms{\cxiu}+\hf\norms{\cxip}\le C (\norms{\cxiu} + \hot).
\end{equation*}
Thus, the optimal error estimates hold for Dirichlet boundary conditions. This completes the proof of Theorem \ref{thm-oeeD-44}.
\end{proof}}

\section{Extension to symmetrizable flux Jacobian and the fully nonlinear cases}\label{sc3}
\subsection{Optimal error estimates for the symmetrizable flux Jacobian}\label{sc31}
\subsubsection{A new projection}
For the case with symmetrizable flux Jacobian, the projection errors of the prime variable $\bs u$ for the convective and diffusive terms cannot be simultaneously eliminated when simply choosing the GGR projection $\bphf\bs u$, as the numerical fluxes are different. To cancel the leading projection errors, inspired by \cite{CMZ2017}, we introduce a modified projection $\pro\bs p=[\textcolor{blue}{\prot \po},\dots,\textcolor{blue}{\prot \xpm}]^\trs$ for the auxiliary variable $\bs p=(p_1, \ldots, p_m)^\trs$ with each component \textcolor{blue}{$\prot p_i$} $(i = 1, \ldots, m)$ satisfying for  $j=1,\dots,N$
\begin{subequations}\label{pp-s}
\begin{align}
\intj{(p_i - \textcolor{blue}{\prot p_i})v_h}&=0,\quad \forall v_h\in P^{k-1}(I_j),\label{pps1}\\
(\textcolor{blue}{\prot p_i} - p_i)_{j-\hf}^{\thgkh} & = c^i_j,
\end{align}
\end{subequations}
where $c^i_j$ is the $i$-th component of $-\cbahfi\vxl{\big(\cfpu(\bs u)\cetuh)}$ with $\cbahfi=diag(\ao^{-\hf},\dots,\am^{-\hf})$ and $\cetu = \bs u - \bphf \bs u$ that is known. Clearly, the above projection implies the following relationship for projection errors, which is essential to the estimate of $\Pi_3+\Pi_5$ in the proof of Theorem \ref{thm-oees-4},
\begin{subequations}
\begin{align}
\intj{\vht(\bs p-\pro\bs p)}&=0,\quad \forall\vh\in\bpkij,\label{pn1}\\
\bs A^\hf(\bs p-\pro\bs p)_{j+\hf}^{\thgkh}&=\big(\cfpu(\bs u)\cetuh\big)_{j+\hf},\quad\forall j=1,\dots,N. \label{pn2}
\end{align}
\end{subequations}
Similar to the scalar case \cite{CMZ2017}, we can derive the following optimal approximation property for the above modified projection.
\begin{lemma}\label{lmm-proeup-6}
Assume that $\bs u\in H^{k+2} $ is periodic. Then, the above defined projection is unique and the following approximation property holds
\begin{equation}\label{ap}
	\norm{\bs p-\pro\bs p}\le C\ho,
\end{equation}
where the positive constant $C$ is independent of h.
\end{lemma}
\begin{proof}
We first prove the existence and uniqueness of $\pro\bs p$.
Let $\bs E=\pro\bs p-\bphz\bs p=[\textcolor{blue}{\prot\po-\bphzt\po},\dots,\textcolor{blue}{\prot\xpm-}$ $\textcolor{blue}{\bphzt\xpm}]^\trs$, where \textcolor{blue}{$\bphzt$} is the GGR projection defined in \eqref{pro-s2}. If we can prove the existence and uniqueness of $\bs E = (E^1, \ldots, E^m)^\trs$, then the  projection $\pro\bs p$ will exist and is unique. By the definition of the two projections \textcolor{blue}{$\bphzt p_i,$ $\prot p_i ~(i = 1, \ldots, m)$} in \eqref{pro-s2} and \eqref{pp-s}, we have
\begin{subequations}
\begin{align}\label{tproj}
\intj{E^i v_h}&=0,\quad \forall v_h\in  P^{k-1}(I_j),\\\label{tprojb}
(E^i)_{j-\hf}^{\thgkh}&=c^i_j,\quad j=1,\dots,N.
\end{align}
\end{subequations}
Denote the restriction of $E^i$ $(i = 1, \ldots, m)$ to each element $I_j$ as
\begin{equation}
E_j^i(x)=\qhl\ajl\legendjl(x)=\qhl\ajl\legendl(s),
\end{equation}
where $P_\ell(s)$ is the $\ell$th-order orthogonal Legendre polynomials on $[-1,1]$ with $s=\frac{2(x-x_j)}{h_j}$. Equality \eqref{tproj} and the orthogonality of Legendre polynomials yield
\begin{equation*}
\ajl=0,\quad \ell=0,\dots,k-1,\quad j=1,\dots,N.
\end{equation*}
Therefore, $E^i_j(x)=\ajk P_k(s).$ It follows from \eqref{tprojb} that
\begin{equation*}
\thg\cajok+(-1)^k \theta \cajk=  c^i_j.
\end{equation*}
For periodic boundary conditions, it can be written as a linear system
\begin{equation}
\bs B\akw=\bs c^i,
\end{equation}
where $\bs B=circ(\theta(-1)^k,0,\dots,0,\thg)$ is an $N\times N$ circulant matrix with first row $(\theta(-1)^k,0,\dots,0,\thg)$, and $\akw=[\aok,\dots,\ank]^\trs$, $\bs c^i=[c^i_1,\dots,c^i_N]^\trs$. It is easy to compute the determinant of $\bs B$
\begin{equation*}
	\det {\bs B}=\big(\theta(-1)^k\big)^N-(-\thg)^N.
\end{equation*}
Since the determinant of $\bs B$ is always not $0$ when $\theta > \hf$, the matrix $\bs B$ is invertible. This implies that the existence and uniqueness of $E^i_j(x)$. Further, we can obtain the existence and uniqueness of $\bs E$.

Next, we turn to the proof of the approximation result \eqref{ap}. In \cite{DPJ1979}, it is shown that the inverse of a nonsingular circulant matrix is also circulant, and thus
$$
\cbbin=\frac{1}{ \theta(-1)^k \big(1-q^N\big)}circ(1,q,\ldots, q^{N-1}),
$$
where $q=-\frac{\thg}{\theta(-1)^k},$ and $|q|<1$ since $\theta > \hf$. By some simple manipulations, we can find that the $1$- and $\infty$-norms of $\cbbin$ are equal and satisfy
\begin{equation*}
\normo{\cbbin}=\norminf{\cbbin}\le\frac{1}{ \theta(1-|q|)}.
\end{equation*}
hence the spectral norm satisfies
$$ \norms{\cbbin}\leq\normo{\cbbin}\norminf{\cbbin}\leq\frac{1}{\theta^2(1-|q|)^2}.
$$
Using the approximation property of $\bphf \bs u$ in \eqref{pru}, we deduce that
\begin{align*}
\norms{\akw}&=\norms{\cbbin\bs c^i}\leq\norms{\cbbin}\norms{\bs c^i}\leq C\norms{\bs c^i}\\
&\leq C h^{2k+1}\norm{\bs u}_{k+1}^2,
\end{align*}
where the positive constant $C$ is independent of $h$. Then,
$$ \norms{E^i}=\qh \intj{ \big( E_j^i(x) \big)^2}= \qh(\ajk)^2\norms{\legendjk(x)}_{I_j}=\qh\frac{h_j(\ajk)^2}{2k+1}\leq Ch\norms{\akw}.
$$
A combination of above two inequalities produces the optimal approximation result of $\norm {E^i}$. Finally, since $\norms{\bs E}=\qhi\norms{E^i}$ and using the approximation property of $\bphz\bs p$ in \eqref{prp}, we arrive at
$$
	\norm{\bs p-\pro\bs p} \le \norm {\bs p - \bphz\bs p} + \norm{\bs E} \le C\ho.
$$
This finishes the proof of Lemma \ref{lmm-proeup-6}.
\end{proof}

\subsubsection{The optimal error estimate}\label{sc312}
\begin{theorem}\label{thm-oees-4}
Assume that the exact solution $\bs u$ of \eqref{eqhs} is sufficiently smooth. Let $\uh$ be the LDG solution with fluxes \eqref{flx-s} and \eqref{flux1}. For a quasi-uniform mesh and $k \ge 1$, we have, for any $T>0$, the following error estimates
\begin{equation}\label{s-o}
	\norm{\bs u(T)-\uh(T)} \leq C\ho,
\end{equation}
where the positive constant $C$ is independent of $h$.
\end{theorem}
\begin{proof}
By Galerkin orthogonality, the error equations are
%\begin{subequations}\label{erreq12}
\begin{align*}
\intqj{\vht\ceupt}&=\hqjbh{\bs f(\bs u)-\bs f(\uh)}{\vh}-\hqjt{\cbahf\cep}{\vh},\\
	\intqj{\wht\cep}&=-\hqjo{\cbahf\ceu}{\wh},\quad\forall (\vh,\wh) \in \cbvh\times\cbvh.
\end{align*}
%\end{subequations}
Further, if we decompose $\ceu, \cep$ into $\ceu=(\bs u-\bphf\bs u) + (\bphf\bs u-\uh) \triangleq \cetu + \cxiu$ and $\cep=(\bs p-\pro\bs p) + (\pro\bs p-\ph)\triangleq \cetp + \cxip$, we get
\begin{subequations}\label{erreq22}
	\begin{align}
	\intqj{\vht\cxiupt}&=-\intqj{\vht\cetupt}+\hqjbh{\bs f(\bs u)-\bs f(\uh)}{\vh}-\hqjt{\cbahf\cxip}{\vh}\notag\\
&\ind-\hqjt{\cbahf\cetp}{\vh},\\ \intqj{\wht\cxip}&=-\intqj{\wht\cetp}-\hqjo{\cbahf\cxiu}{\wh}-\hqjo{\cbahf\cetu}{\wh},
\end{align}
\end{subequations}
which hold for all $(\vh,\wh) \in \cbvh\times\cbvh$.

Taking $\vh=\bs Q_r\cxiu,\wh=\bs Q_r\cxip$ in \eqref{erreq22}, we have
\begin{align*}
		\intqj{\cxiut\bs Q_r\cxiupt}&=-\intqj{\cxiut\bs Q_r\cetupt}+\hqjbh{\bs f(\bs u)-\bs f(\uh)}{\bs Q_r\cxiu}-\hqjt{\cbahf\cxip}{\bs Q_r\cxiu}\\
		&\ind-\hqjt{\cbahf\cetp}{\bs Q_r\cxiu},\\
		\intqj{\cxiptr\bs Q_r\cxip}&=-\intqj{\cxiptr\bs Q_r\cetp}-\hqjo{\cbahf\cxiu}{\bs Q_r\cxip}-\hqjo{\cbahf\cetu}{\bs Q_r\cxip},
\end{align*}
where $\bs Q_r = \bs Q_{j+\hf}$ denotes the value of the \textcolor{blue}{symmetric positive definite} matrix $\bs Q = \bs v'_{\bs u}$ at $x_{j+\hf}$; see Sec. \ref{sc213}. Summing up above two identities, we have
\begin{align}\label{erreq42}
\hf\dt\norms{\bs Q_r^\hfhk\cxiu}+\norms{\bs Q_r^\hfhk\cxip}=
\Pi_1 +  \Pi_2 + \Pi_3 + \Pi_4 +\Pi_5,
\end{align}
where
\begin{align*}
	\Pi_1 & =-\intqj{\cxiut\bs Q_r\cetupt}-\intqj{\cxiptr\bs Q_r\cetp}+\hf\intqj{\cxiut(\bs Q_r)_t\cxiu},\\
	\Pi_2 & =\hqjbh{\cfpu(\bs u)\cxiu}{\bs Q_r\cxiu},\\
	\Pi_3 & =\hqjbh{\cfpu(\bs u)\cetu}{\bs Q_r\cxiu},\\
	\Pi_4 &=- \hqjbh{\bs e_{\bs u}^\trs \bs H \bs e_{\bs u}}{\bs Q_r\cxiu},\\
	\Pi_5 & =-\hqjt{\cbahf\cxip}{\bs Q_r\cxiu}-\hqjo{\cbahf\cxiu}{\bs Q_r\cxip}-\hqjt{\cbahf\cetp}{\bs Q_r\cxiu}-\hqjo{\cbahf\cetu}{\bs Q_r\cxip}
\end{align*}
will be estimated separately.

Let us first consider $\Pi_1$. By the Cauchy-Schwarz inequality and Young's inequality, we get
\begin{subequations}
	\begin{align}\label{eres12}\notag
		\Pi_1 &\leq C\ho (\norm{\cxiu}+\norm{\bs Q_r^\hfhk\cxip})+ C\norms{\cxiu}\\
		&\le C\norms{\cxiu}+\hf\norms{\bs Q_r^\hfhk\cxip}+C\hot.
\end{align}
Using integration by parts, $\ceu=\bs u-\uh=\bs R_r(\bs z-\bs z_h)\triangleq \bs R_r\cet_{\bs z} + \bs R_r \cxiz$ and $\cfpu(\bs u_{j+\hf})\bs R_{j+\hf}=\bs R_{j+\hf}\bs \Lambda_{j+\hf}$, we arrive at
\begin{align}\label{eres22}\notag
\Pi_2   & =\qh\intj{\cxiuxt\vxr{\bs Q}\cfpu(\bs u)\cxiu}+\qh\vxr{\jumpxut}\vxr{\bs Q}\vxr{\big(\cfpu(\bs u)\cxiuh\big)}\\\notag
		&=\qh\intj{\cxiuxt\vxr{\bs Q}\cfpu(\vxr{\bs u})\cxiu}+\qh\vxr{\jumpxut}\vxr{\bs Q}\vxr{\big(\cfpu(\bs u)\cxiuh\big)}\\\notag
		&\ind +\qh\intj{\cxiuxt\vxr{\bs Q}(\cfpu(\bs u)-\cfpu(\vxr{\bs u}))\cxiu}\\\notag
		&\leq\qh\intj{\cxiuxt\vxr{\bs Q}^\hfhk\vxr{\bs K}\vxr{\bs Q}^\hfhk\cxiu}+\qh\vxr{\jumpxut}\vxr{\bs Q}\vxr{\big(\cfpu(\bs u)\cxiuh\big)}
		 +C\norms{\cxiu}\\\notag
		&=\qh\vxr{\bigg(\jumpxut\bs Q^\hfhk\bs K\bs Q^\hfhk\big(\cxiuh-\xiuave\big)\bigg)}+C\norms{\cxiu}\\\notag
		&=\qh\vxr{\bigg(\jumpxut\bs Q\cfpu(\bs u)\big(\cxiuh-\xiuave\big)\bigg)}+C\norms{\cxiu}\\\notag
		&=\qh\vxr{\bigg(\jumpxizt\cbrt\bs Q\cfpu(\bs u)\bs R\big(\cxizh-\xizave\big)\bigg)}+C\norms{\cxiu}\\\notag
		&=\qh\vxr{\bigg(\jumpxizt\cbrt\bs Q\bs R\bs \Lambda\big(\cxizh-\xizave\big)\bigg)}+C\norms{\cxiu}\\\notag
		&=-(\theta-\hf)\qh\vxr{\bigg(\jumpxizt\cbrt\bs Q\bs R|\bs \Lambda|\jumpxiz\bigg)}+C\norms{\cxiu}\\
		&\leq C \norms{\cxiu},
\end{align}
since $\theta > \hf$, where we have also used \eqref{qf}, the fact that $\bs K$ is symmetric and $\cbrt\bs Q\bs R|\bs \Lambda|$ is positive semi-definite.
By the local linearization of $\cfpu(\bs u)$ at $\bs u_j$, we first rewrite  $\Pi_3$ as
\begin{align*}
\Pi_3 & = \qh\intj{\cxiuxt\vxr{\bs Q}\cfpu(\bs u)\cetu}+\qh\vxr{\jumpxut}\vxr{\bs Q}\vxr{\big(\cfpu(\bs u)\cetuh\big)}\\
		& = \qh\intj{\cxiuxt\vxr{\bs Q}\cfpu(\cuc)\cetu}+\qh\intj{\cxiuxt\vxr{\bs Q}\big(\cfpu(\bs u)-\cfpu(\cuc)\big)\cetu}\\
		&\ind +\qh\vxr{\jumpxut}\vxr{\bs Q}\vxr{\big(\cfpu(\bs u)\cetuh\big)}.
\end{align*}
It is easy to show for high order term $\Pi_4$ that
\begin{align}\label{eres42}\notag
\Pi_4 &\le C \norminf{\ceu}\big(\norm {\bs e_{\bs u}} \norm {(\cxiu)_x} +\normb {\bs e_{\bs u}} \normb \cxiu \big) \notag\\
&\le C h^{-1} \norminf{\ceu}\big(\norm \cxiu +\norm \cetu + h^\hf \normb \cetu\big)\norm \cxiu \notag\\
&\le C (\norms{\cxiu} + \hot).
\end{align}
For $\Pi_5$, by the generalized skew-symmetry property in Lemma \ref{lmm-dgvr-5}, the symmetric property of $\bs Q_r\cbahf$ and the projection property in \eqref{op-pro} as well as \eqref{pps1}, we have
\begin{align}\label{eres52}\notag
		\Pi_5&=-\hqjt{\cbahf\cetp}{\bs Q_r\cxiu}\\\notag
		&=-\qh\intj{\cxiuxt\vxr{\bs Q}\cbahf\cetp}-\qh\vxr{\jumpxut}\vxr{\bs Q}\cbahf\vxr{(\cetpthg)}\\
		&=-\qh\vxr{\jumpxut}\vxr{\bs Q}\cbahf\vxr{(\cetpthg)}.
\end{align}
Consequently, by $\bs A^\hf(\cetpthg)_{j+\hf}=\big(\cfpu(\bs u)\cetuh\big)_{j+\hf}$ in \eqref{pn2} and the projection property in \eqref{op-pro}, we conclude that
\begin{align}\label{eres62}\notag
		\Pi_3 + \Pi_5&=\qh\intj{\cxiuxt\vxr{\bs Q}\cfpu(\cuc)\cetu}+\qh\intj{\cxiuxt\vxr{\bs Q}\big(\cfpu(\bs u)-\cfpu(\cuc)\big)\cetu}\\
& = \qh\intj{\cxiuxt\vxr{\bs Q}\big(\cfpu(\bs u)-\cfpu(\cuc)\big)\cetu} \notag \\
		&\leq C( \norms{\cxiu} + \hot).
\end{align}
\end{subequations}
Collecting \eqref{eres12}--\eqref{eres62} into \eqref{erreq42}, we have
\begin{equation*}
	\hf\dt\norms{\bs Q_r^\hf\cxiu}+\hf\norms{\bs Q_r^\hfhk\cxip}\le C\norms{\cxiu} + C\hot.
\end{equation*}
The Gronwall's inequality together with the equivalence of the norms $\norm{\cdot}$ and $\norm{\bs Q_r^\hf\cdot}$ (since $\bs Q_r^\hf$ is \textcolor{blue}{symmetric positive definite} uniformly) leads to the expected optimal error estimate \eqref{s-o}. This completes the proof of Theorem \ref{thm-oees-4}.
\end{proof}

%\begin{remark} \label{rek-matx-2}
%In the proof of Theorem \ref{thm-oees-4}, to use the skew-symmetry property of the DG operator in the estimate of $\Pi_5$, we need the matrix $\bs Q_r\cbahf$ to be symmetric. When $\bs Q_r\cbahf$ is not symmetric, we can also obtain the optimal error estimate by rewriting the original system
%$\cut+\bs f(\bs u)_x=(\bs A(\bs u)\bs u_{x})_{x}$
 %into an equivalent system
 %$\bs Q^{-1} \bs v_t +\bs f(\bs v)_x=(\widetilde {\bs A}(\bs v){\bs v}_{x})_{x}$
 %using the symmetrizable theory.
%\end{remark}

\subsection{Optimal error estimates for the fully nonlinear diffusive case}\label{sc32}
In this section, we display how to obtain the optimal error estimate for the following fully nonlinear convection-diffusion equations
\begin{subequations}\label{fxxd}
\begin{align}
		\cut+\bs f(\bs u)_x&=(\bs A(\bs u)\bs u_{x})_{x}, \quad (x,t)\in \qj\times(0,T],\\
		\bs u(x,0)&=\bs u_0(x),  \qquad \quad~~\! x\in \qj,
\end{align}
\end{subequations}
where \textcolor{blue}{$\bs A(\bs u)=(a_{ij}(\bs u))_{i=1,\dots,m}^{j=1,\dots,m}$ is symmetric positive semi-definite, which implies that there exists a symmetric positive semi-definite matrix $\bs B(\bs u)$ such that $\bs B^2(\bs u)=\bs A(\bs u)$.}

We mainly show the proof for \eqref{fxxd} with \textcolor{blue}{symmetric positive definite} flux Jacobian; the case with symmetrizable flux Jacobian is discussed in Remark \ref{rem-diss-3}.
The above system can be written as an equivalent system
\begin{equation*}
	\cut+\bs f(\bs u)_x=\textcolor{blue}{{(\bs B(\bs u)\bs p)}_x},\quad \bs p=\textcolor{blue}{\bs B(\bs u)\bs u_x=\bs g(\bs u)_x},
\end{equation*}
\textcolor{blue}{where the function $\bs g(\bs u)$ satisfies $\bs g'(\bs u)=\bs B(\bs u)$.}
The LDG scheme is: find $(\uh,\ph)\in\cbvh\times\cbvh$ such that
\begin{subequations}
\begin{align*}
		\intj{\vht(\uh)_t}&=\hjbh{\bs f(\uh)}{\vh}-\hjt{\textcolor{blue}{\bs B(\uh)\ph}}{\vh},\\
		\intj{\wht\ph}&=-\hjo{\textcolor{blue}{\bs g(\uh)}}{\wh}
\end{align*}
\end{subequations}
hold for any $(\vh,\wh)\in\cbvh\times\cbvh$ and $j=1,\dots,N$. Different from the linear case, we need to redefine numerical fluxes for diffusive terms, i.e.,
\begin{equation*}
	\textcolor{blue}{\hat{\bs g}(\uh)=\bs g^{\thkh}(\uh)},\quad \cphh=\phthg,
\end{equation*}
and
\begin{equation*}
	\textcolor{blue}{\hat{\bs B}(\uh)=\frac{\jump{\bs g(\uh)}\jump{\uh}^\trs}{\big\|\jump{\uh}\big\|_M^2},}
\end{equation*}
where the subscript $j+\hf$ is omitted.

The projection for the prime variable $\bs u$ can be chosen as that in Sect. \ref{sc22}, namely $\bphf\bs u$. However, since $\bs A(\bs u)$ is nonlinear, a modified projection of $\bs p$, denoted by $\pw \bs p$,  shall be introduced satisfying
\begin{align*}
	\intj{\vht(\bs p-\pw\bs p)}&=0,\quad \forall\vh\in\bpkij,\\
	\Big(\textcolor{blue}{\bs B(\bs u)}(\bs p - \pw \bs p)^{\thgkh}\Big)_{j-\hf}&=-(\theta-\hf)\Big(\textcolor{blue}{\bs B'(\bs u)}\jump{\cetu}\textcolor{blue}{\bs g(\bs u)_x}\Big)_{j-\hf},\quad j=1,\dots,N,
\end{align*}
where $\bs B'(\bs u)$ is an $m\times m \times m$ tensor. 

If we decompose $\ceu, \cep$ into $\ceu=(\bs u-\bphf\bs u) + (\bphf\bs u-\uh) \triangleq \cetu + \cxiu$ and $\cep=(\bs p-\pw\bs p) + (\pw\bs p-\ph) \triangleq  \cetp + \cxip$, and take $\vh=\cxiu,\wh=\cxip$, we can obtain the following error equations
	\begin{align*}
	\intqj{\cxiut\cxiupt}&=-\intqj{\cxiut\cetupt}+\hqjbh{\bs f(\bs u)-\bs f(\uh)}{\cxiu}\\
	&\ind-\hqjt{\textcolor{blue}{\bs B(\bs u)}\bs p-\textcolor{blue}{\bs B(\uh)}\ph}{\cxiu},\\
	\intqj{\cxiptr\cxip}&=-\intqj{\cxiptr\cetp}-\hqjo{\textcolor{blue}{\bs g(\bs u)-\bs g(\uh)}}{\cxip}.
	\end{align*}
Adding above two equations, we have
\begin{align*}
	\hf\dt\norms{\cxiu}+\norms{\cxip}=\Theta_1 +  \Theta_2 + \Theta_3 + \Theta_4 +\Theta_5,
\end{align*}
where
\begin{align*}
	\Theta_1 & =-\intqj{\cxiut\cetupt}-\intqj{\cxiptr\cetp},\\
	\Theta_2 & =\hqjbh{\cfpu(\bs u)\cxiu}{\cxiu},\\
	\Theta_3 & =\hqjbh{\cfpu(\bs u)\cetu}{\cxiu},\\
	\Theta_4 &=- \hqjbh{\ceut\bs H\ceu}{\cxiu},\\
	\Theta_5 & =-\hqjt{\textcolor{blue}{\bs B(\bs u)}\bs p-\textcolor{blue}{\bs B(\uh)}\ph}{\cxiu}-\hqjo{\textcolor{blue}{\bs g(\bs u)-\bs g(\uh)}}{\cxip}.
	\end{align*}
By the choice of numerical fluxes and projections, we can see that the estimates of $\Theta_1$--$\Theta_4$ are exactly the same as that in the proof of Theorem \ref{thm-oeep-3}. So, we only need to consider the term $\Theta_5$. \textcolor{blue}{By Taylor expansion and the mean value theorem, we first have the following expression
	\begin{align*}
	\bs B(\bs u)\bs p-\bs B(\uh)\ph&=(\bs B'(\bs u)\ceu-\bs R_1)\bs p+(\bs B(\bs u)-\bs B'(\bs u)\ceu+\bs R_1)\bs e_p,\\
	\bs B(\bs u)\bs p-\hat{\bs B}(\uh)\hat{\bs p}_h&=(\bs B'(\bs u)\ave{\ceu}-\bs R_2)\bs p+(\bs B(\bs u)-\bs B'(\bs u)\ave{\ceu}+\bs R_2)\bs e_p^{(\bar{\theta})},
	\end{align*}
and
\begin{align*}
    \bs g(\bs u)-\bs g(\uh)&=B(\bs u)\ceu-\bs R_3,\\
	\bs g(\bs u)-\bs g^{\thkh}(\uh)&= B(\bs u)\ceuhth-\bs R_4,
\end{align*}
where $\bs R_i~ (i=1,\dots,4)$ are all high order terms derived from Taylor expansion.}

Using the same argument as that in \cite{Cheng2019} for scalar nonlinear convection-diffusion equations, and by virtue of the newly designed projection $\pw\bs p$ and \textcolor{blue}{the symmetric property of $\bs B(\bs u)$}, we can easily obtain the estimate for $\Theta_5$; it reads,
\begin{equation*}
	\Theta_5 \leq C\norms{\cxiu}+ \hf\norms{\cxip} + C\hot.
\end{equation*}
This, together with the estimates of $\Lambda_1$--$\Lambda_4$ in Theorem \ref{thm-oeep-3}, produces the desired optimal error estimates.

\begin{remark}\label{rem-diss-3}
For the case when the flux Jacobian is symmetrizable, a modified projection $\pb\bs p$ is defined to satisfy for $j= 1, \ldots, N$
\begin{align*}
	\intj{\vht(\bs p-\pb\bs p)}&=0,\quad \forall\vh\in\bpkij,\\
	\Big(\textcolor{blue}{\bs B(\bs u)}(\bs p - \pb \bs p)^{\thgkh}\Big)_{j-\hf}&=\Big(-(\theta-\hf)\textcolor{blue}{\bs B'(\bs u)}\jump{\cetu}\textcolor{blue}{\bs g(\bs u)_x}+\cfpu(\bs u)\cetuh\Big)_{j-\hf}.
\end{align*}
The optimal error estimates can be derived analogously, and details are omitted.
\end{remark}

\section{Numerical Experiments}\label{sc4}
In this section, we show several numerical examples to validate optimal error estimates of the LDG method using generalized numerical fluxes for nonlinear convection-diffusion systems. We use the explicit third order total variation diminishing Runge-Kutta time discretization and uniform meshes \textcolor{blue}{with $\Delta t = CFL_k~h^2$, and $CFL_k$ for $P^k$ polynomials are given in Table \ref{cfl}}.  Systems with different boundary conditions, long time simulations, and degenerate equations with discontinuous initial data  are numerically tested to show the sharpness of theoretical results and the efficiency of LDG methods with generalized fluxes.
\vspace{-0.3em}
\begin{table}[htp!]
\newsavebox{\tablebox}
	\caption{\label{cfl}\bh{ $CFL$ numbers used in numerical examples.}} \centering
	\vspace{-0.4em}
	\bigskip
	\begin{lrbox}{\tablebox}
			\begin{tabular}{ccccccccc}
     \hline
      Table/Figure&Table \ref{table2} & Table \ref{table3}&Table \ref{table4} &Table \ref{table5} & Table \ref{table6} &Table \ref{table7} & Figure \ref{figure1}& Figure \ref{figure2} \\ \hline
     $CFL_0$& 0.005 & 0.6 &0.0001 & 0.005 & 0.005 &0.005 & -- & -- \\
      $CFL_1$& 0.005 & 0.01 &0.0001 & 0.005 & 0.005 &0.005 & -- &0.05 \\
       $CFL_2$& 0.005 & 0.01 &0.00008 & 0.005 & 0.005 &0.005 & 0.005 & 0.01 \\
       $CFL_3$& 0.002 & 0.002 &0.00002 & 0.003 & 0.003 &0.0005 & -- & -- \\
       $CFL_4$& 0.001 & 0.002 &0.00001 & 0.002 & 0.0005 &0.0001 & -- & -- \\
     \hline
   \end{tabular}
	\end{lrbox}
	\scalebox{0.91}{\usebox{\tablebox}}
\end{table}
\begin{example}\label{ex2}
Consider the system with a linear diffusion term and periodic boundary conditions
\begin{align*}
	\cut+\bs f(\bs u)_x&=\bs A\bs u_{xx}+\bs g,\quad (x,t)\in(0,2\pi)\times(0,T]\\
	\bs u(x,0)&=\bs u_0(x),\quad\quad\quad x\in (0,2\pi),
\end{align*}
where $\bs u=(\uo,\ut,\uth)^\trs$, $\bs f(\bs u)=(\uo^3,\ut^3,\uth^3)^\trs$,
	$\bs A=diag(1,1,1)$. The source term and initial condition are suitably chosen such that the exact solution is
	\begin{equation}\label{exact:u} \bs u=(\exp(-t)\sin(2x+t),\exp(-t)\sin(2x-t),\exp(-2t)\sin(x+t))^\trs.
	\end{equation}	
\end{example}
The $L^2$ errors and numerical orders for Example \ref{ex2} are given in Table \ref{table2}, from which expected optimal order can be observed. In addition, the cases of a convection dominated problem, e.g., $\bs A=diag(10^{-4},10^{-4},10^{-4})$ and a strongly anisotropic problem, e.g., $\bs A=diag(100,100,100)$ are considered, for which the source terms are suitably chosen such that the exact solution \eqref{exact:u} is unchanged. The results shown in Tables \ref{table3} and \ref{table4} exhibit the desired optimal $(k+1)$th order, demonstrating that the theoretical results also hold for both convection dominated problems and strongly anisotropic problems.\vspace{0.3em}

\begin{table}[htp!]
	\caption{\label{table2} $L^2$ errors and orders for Example \ref{ex2} with different $\theta$, $\bs A=diag(1,1,1)$, $T=1$.} \centering
	\vspace{-0.4em}
	\bigskip
	\begin{lrbox}{\tablebox}
		\begin{tabular}{cccccccccc}
			\hline
			& &\multicolumn{2}{c}{$\theta = 0.8$}&
			&\multicolumn{2}{c}{$\theta = 1.0$}&&\multicolumn{2}{c}{$\theta = 1.2$}\\
			\cline{3-4}   \cline{6-7} \cline{9-10}
			&         {$N$}      & $L^2$ error& Order &~~~& $L^2$ error & Order & ~~~ &  $L^2$  error  & Order \\
			\hline
			\multirow{4}{*}{$P^0$}
			&$10$&   3.77E-01  &  --   &&   3.44E-01  &  --   &&  4.03E-01  &  --  \\
			&$20$&   1.77E-01  & 1.09  &&   1.86E-01  & 0.89  &&  2.18E-01  & 0.89\\
			&$40$&   9.26E-02  & 0.93  &&   9.80E-02  & 0.92  &&  1.09E-01  & 1.00 \\
			&$80$&   4.87E-02  & 0.93  &&   5.05E-02  & 0.96  &&  5.37E-02  & 1.02 \\ \vspace{-0.8em}
			\\
			\multirow{4}{*}{$P^1$}
			&$10$&   9.87E-02  &  --   &&   8.85E-02  &  --   &&  8.28E-02  &  --  \\
			&$20$&   2.89E-02  & 1.77  &&   2.22E-02  & 2.00  &&  1.91E-02  & 2.12\\
			&$40$&   7.78E-03  & 1.90  &&   5.55E-03  & 2.00  &&  4.66E-03  & 2.03 \\
			&$80$&   1.99E-03  & 1.97  &&   1.39E-03  & 2.00  &&  1.16E-03  & 2.01 \\ \vspace{-0.8em}
			\\
			\multirow{4}{*}{$P^2$}
			&$10$&   8.01E-03  &  --   &&   8.72E-03  &  --   &&  9.21E-03  &  --  \\
			&$20$&   9.10E-04  & 3.14  &&   1.11E-03  & 2.98  &&  1.31E-03  & 2.81\\
			&$40$&   1.11E-04  & 3.03  &&   1.39E-04  & 2.99  &&  1.71E-04  & 2.94 \\
			&$80$&   1.38E-05  & 3.01  &&   1.74E-05  & 3.00  &&  2.17E-05  & 2.98 \\ \vspace{-0.8em}
			\\
			\multirow{4}{*}{$P^3$}
			&$10$&   7.30E-04  &  --    &&  6.78E-04  &  --   &&  6.43E-04  &  --  \\
			&$20$&   5.51E-05  & 3.73  &&   4.29E-05  & 3.98  &&  3.76E-05  & 4.10\\
			&$40$&   3.70E-06  & 3.90  &&   2.69E-06  & 4.00  &&  2.31E-06  & 4.03 \\
			&$80$&   2.36E-07  & 3.97  &&   1.68E-07  & 4.00  &&  1.44E-07  & 4.01 \\ \vspace{-0.8em}
			\\
			\multirow{4}{*}{$P^4$}
			&$10$&   2.62E-05  &  --    &&  3.12E-05  &  --   &&  3.48E-05  &  --  \\
			&$20$&   6.51E-07  & 5.33  &&   9.88E-07  & 4.98  &&  1.29E-06  & 4.75\\
			&$40$&   1.91E-08  & 5.09  &&   3.10E-08  & 4.99  &&  4.27E-08  & 4.92 \\
			&$80$&   5.86E-10  & 5.02  &&   9.71E-10  & 5.00  &&  1.35E-09  & 4.98 \\ \vspace{-1.2em}
			\\
			\hline
		\end{tabular}
	\end{lrbox}
	\scalebox{1.04}{\usebox{\tablebox}}
\end{table}

\vspace{-0.3em}

\begin{table}[htp!]
	\caption{\label{table3} $L^2$ errors and orders for Example \ref{ex2} with a convection dominated problem and different $\theta$, $\bs A=diag(10^{-4},10^{-4},10^{-4})$, $T=1$.} \centering
	\vspace{-0.4em}
	\bigskip
	\begin{lrbox}{\tablebox}
		\begin{tabular}{cccccccccc}
			\hline
			& &\multicolumn{2}{c}{$\theta = 0.8$}&
			&\multicolumn{2}{c}{$\theta = 1.0$}&&\multicolumn{2}{c}{$\theta = 1.2$}\\
			\cline{3-4}   \cline{6-7} \cline{9-10}
			&         {$N$}      & $L^2$ error& Order &~~~& $L^2$ error & Order & ~~~ &  $L^2$  error  & Order \\
			\hline
			\multirow{4}{*}{$P^0$}
			&$40$&   4.69E-01  &  --   &&   4.69E-01  &  --   &&  4.69E-01  &  --  \\
			&$80$&   2.63E-01  & 0.84  &&   2.63E-01  & 0.84  &&  2.63E-01  & 0.84\\
			&$120$&   1.85E-01  & 0.86  &&   1.85E-01  & 0.86  &&  1.85E-01  & 0.86 \\
			&$160$&   1.44E-01  & 0.88  &&   1.44E-01  & 0.88  &&  1.44E-01  & 0.88 \\ \vspace{-0.8em}
			\\
			\multirow{4}{*}{$P^1$}
			&$20$&   6.56E-03  &  --   &&   6.55E-03  &  --   &&  6.53E-03  &  --  \\
			&$40$&   1.46E-03  & 2.16  &&   1.46E-03  & 2.17  &&  1.45E-03  & 2.17\\
			&$80$&   6.26E-04  & 2.09  &&   6.23E-04  & 2.10  &&  6.17E-04  & 2.11 \\
			&$160$&  3.47E-04  & 2.06  &&   3.45E-04  & 2.06  &&  3.40E-04  & 2.07 \\ \vspace{-0.8em}
			\\
			\multirow{4}{*}{$P^2$}
			&$20$&   4.27E-04  &  --   &&   4.28E-04  &  --   &&  4.30E-04  &  --  \\
			&$40$&   4.87E-05  & 3.13  &&   4.90E-05  & 3.13  &&  4.95E-05  & 3.12\\
			&$80$&   1.17E-05  & 3.51  &&   1.19E-05  & 3.50  &&  1.21E-05  & 3.47 \\
			&$160$&   3.86E-06  & 3.86  &&   3.94E-06  & 3.83  &&  4.06E-06  & 3.80 \\ \vspace{-0.8em}
			\\
			\multirow{4}{*}{$P^3$}
			&$20$&   2.88E-06  &  --    &&  2.86E-06  &  --   &&  2.82E-06  &  --  \\
			&$40$&   1.69E-07  & 3.73  &&   1.67E-07  & 3.98  &&  1.62E-07  & 4.12\\
			&$80$&   3.33E-08  & 3.90  &&   3.26E-08  & 4.00  &&  3.15E-08  & 4.05 \\
			&$160$&   1.06E-08  & 3.97  &&   1.03E-08  & 4.00  &&  9.86E-09  & 4.04 \\ \vspace{-0.8em}
			\\
			\multirow{4}{*}{$P^4$}
			&$20$&   9.25E-08  &  --    &&  9.34E-08  &  --   &&  9.48E-08  &  --  \\
			&$40$&   1.81E-09  & 5.67  &&   1.87E-09  & 5.64  &&  1.96E-09  & 5.59\\
			&$80$&   1.72E-10  & 5.81  &&   1.81E-10  & 5.76  &&  1.96E-10  & 5.68 \\
			&$160$&   3.43E-11  & 5.60  &&   3.69E-11  & 5.54  &&  4.29E-11  & 5.45 \\ \vspace{-1.2em}
			\\
			\hline
		\end{tabular}
	\end{lrbox}
	\scalebox{1.04}{\usebox{\tablebox}}
\end{table}

\vspace{-0.3em}

\begin{table}[htp!]
	\caption{\label{table4} $L^2$ errors and  orders for Example \ref{ex2} with a strongly anisotropic problem and different $\theta$, $\bs A=diag(100,100,100)$, $T=1$.} \centering
	\vspace{-0.4em}
	\bigskip
	\begin{lrbox}{\tablebox}
		\begin{tabular}{cccccccccc}
			\hline
			& &\multicolumn{2}{c}{$\theta = 0.8$}&
			&\multicolumn{2}{c}{$\theta = 1.0$}&&\multicolumn{2}{c}{$\theta = 1.2$}\\
			\cline{3-4}   \cline{6-7} \cline{9-10}
			&         {$N$}      & $L^2$ error& Order &~~~& $L^2$ error & Order & ~~~ &  $L^2$  error  & Order \\
			\hline
			\multirow{4}{*}{$P^0$}
			&$10$&   5.19E-01  &  --   &&   3.51E-01  &  --   &&  3.51E-01  &  --  \\
			&$20$&   1.91E-01  & 1.44  &&   1.70E-01  & 1.04  &&  1.75E-01  & 1.01\\
			&$30$&   1.19E-01  & 1.17  &&   1.13E-01  & 1.01  &&  1.15E-01  & 1.04 \\
			&$40$&   8.71E-02  & 1.08  &&   8.45E-02  & 1.01  &&  8.53E-02  & 1.03 \\ \vspace{-0.8em}
			\\
			\multirow{4}{*}{$P^1$}
			&$10$&   9.44E-02  &  --   &&   8.82E-02  &  --   &&  8.32E-02  &  --  \\
			&$20$&   2.89E-02  & 1.71  &&   2.22E-02  & 1.99  &&  1.90E-02  & 2.13\\
			&$30$&   1.36E-02  & 1.86  &&   9.86E-03  & 2.00  &&  8.33E-03  & 2.04 \\
			&$40$&   7.82E-03  & 1.92  &&   5.55E-03  & 2.00  &&  4.66E-03  & 2.02 \\ \vspace{-0.8em}
			\\
			\multirow{4}{*}{$P^2$}
			&$10$&   8.06E-03  &  --   &&   8.73E-03  &  --   &&  9.19E-03  &  --  \\
			&$20$&   9.12E-04  & 3.14  &&   1.11E-03  & 2.98  &&  1.31E-03  & 2.81\\
			&$30$&   2.66E-04  & 3.04  &&   3.30E-04  & 2.99  &&  4.02E-04  & 2.92 \\
			&$40$&   1.11E-04  & 3.02  &&   1.39E-04  & 3.00  &&  1.72E-04  & 2.96 \\ \vspace{-0.8em}
			\\
			\multirow{4}{*}{$P^3$}
			&$10$&   7.30E-04  &  --    &&  6.77E-04  &  --   &&  6.42E-04  &  --  \\
			&$15$&   1.64E-04  & 3.68  &&   1.35E-04  & 3.97  &&  1.21E-04  & 4.12\\
			&$20$&   5.53E-05  & 3.79  &&   4.29E-05  & 3.99  &&  3.75E-05  & 4.06 \\
			&$25$&   2.34E-05  & 3.85  &&   1.76E-05  & 3.99  &&  1.53E-05  & 4.04 \\ \vspace{-0.8em}
			\\
			\multirow{4}{*}{$P^4$}
			&$10$&   2.63E-05  &  --    &&  3.13E-05  &  --   &&  3.48E-05  &  --  \\
			&$15$&   2.94E-06  & 5.41  &&   4.16E-06  & 4.98  &&  5.18E-06  & 4.69\\
			&$20$&   6.54E-07  & 5.22  &&   9.91E-07  & 4.99  &&  1.29E-06  & 4.82 \\
			&$25$&   2.08E-07  & 5.14  &&   3.25E-07  & 4.99  &&  4.35E-07  & 4.89 \\ \vspace{-1.2em}
			\\
			\hline
		\end{tabular}
	\end{lrbox}
	\scalebox{1.04}{\usebox{\tablebox}}
\end{table}

\begin{example}\label{ex3}
To illustrate the case with long time simulations, consider
\begin{align*}
	\cut+\bs f(\bs u)_x&=\bs A\bs u_{xx}+\bs g,\quad (x,t)\in(0,2\pi)\times(0,T]\\
	\bs u(x,0)&=\bs u_0(x),\quad\quad\quad x\in (0,2\pi)
\end{align*}
with periodic boundary conditions, where $\bs u=(\uo,\ut,\uth)^\trs$, $\bs f(\bs u)=(\uo^3,\ut^3,\uth^3)^\trs,$ $\bs A=diag(1,1,1)$. The source term and initial condition are suitably chosen such that the exact solution is
\begin{align*}
&\bs u=(\exp(-0.01t)\sin(2x+0.1t),\exp(-0.01t)\sin(2x-0.1t),\exp(-0.01t)\sin(x+0.1t))^\trs.
\end{align*}	
\end{example}
The $L^2$ errors for Example \ref{ex3} up to $T=200$ are shown in Figure \ref{figure1}, from which we can see that the LDG scheme exhibits excellent long time behaviors. The magnitude for errors of generalized fluxes is smaller compared to the standard upwind fluxes on the same meshes.

\begin{figure}[!htbp]
	\centering

		\begin{minipage}[t]{0.9\linewidth}
			\centering
			\includegraphics[width=\linewidth]{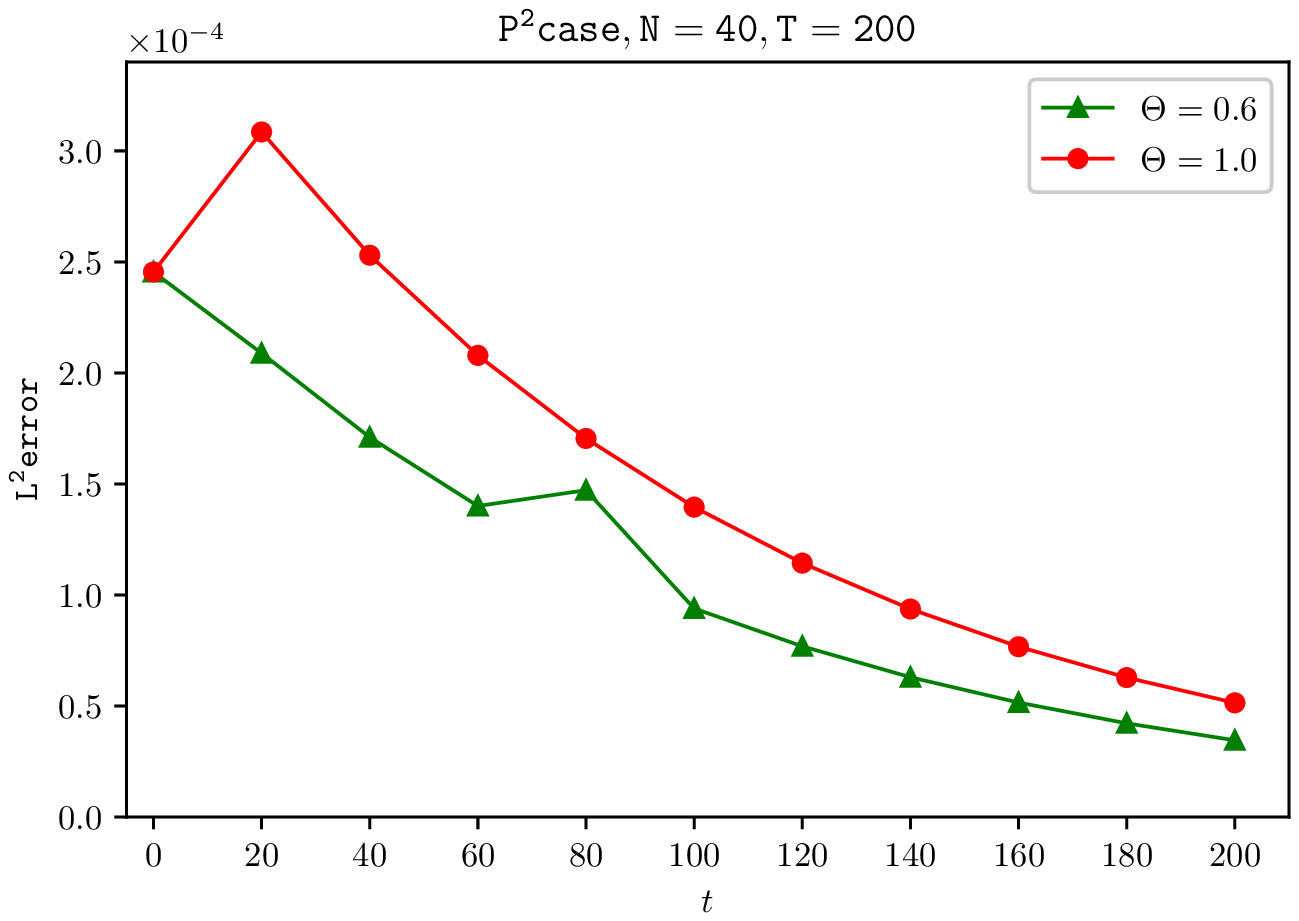}
		\end{minipage}
	\centering
	\caption{Time history of the $L^2$ errors for Example \ref{ex3} }\label{figure1}
\end{figure}
\vspace{-0.3em}
\begin{example}\label{ex4}
	To show the theoretical results also hold for mixed and Dirichlet boundary conditions, consider the system
	\begin{align*}
	\cut+\bs f(\bs u)_x&=\bs A\bs u_{xx}+\bs g,\quad (x,t)\in(0,\pi)\times(0,T]\\
	\bs u(x,0)&=\bs u_0(x),\quad\quad\quad x\in (0,\pi),
	\end{align*}
	where $\bs u=(\uo,\ut)^\trs,\bs f(\bs u)=(\uo^2/2,\ut^2/2)^\trs,\bs A=diag(1,1)$.
	The mixed boundary conditions
	\begin{align*}
	&\uo(0,t)=\exp(-t)\sin(t),\\
	&(\uo)_x(\pi,t)=3\exp(-t)\cos(3\pi+t),\\
	&\ut(0,t)=\exp(-t)\sin(t),\\
	&(\ut)_x(\pi,t)=3\exp(-t)\cos(3\pi+t),
	\end{align*}
	and Dirichlet boundary conditions
	\begin{align*}
	&\uo(0,t)=\exp(-t)\sin(t),\\
	&\uo(\pi,t)=\exp(-t)\sin(3\pi+t),\\
	&\ut(0,t)=\exp(-t)\sin(t),\\
	&\ut(\pi,t)=\exp(-t)\sin(3\pi+t)
	\end{align*}
	are concerned, \textcolor{blue}{respectively}, in which the source term and initial condition are suitably chosen such that the exact solution is
	\begin{equation*}
	\bs u=(\exp(-t)\sin(3x+t),\exp(-t)\sin(3x+t))^\trs.
	\end{equation*}	
\end{example}
The $L^2$ errors and numerical orders for Example \ref{ex4} are given in Tables \ref{table5} and \ref{table6}, indicating that the optimal error estimate results are also valid for mixed and Dirichlet boundary conditions.

\begin{table}[htp!]
	\caption{\label{table5} $L^2$ errors and  orders for Example \ref{ex4} with mixed boundary conditions and different $\theta$, $T=1$.} \centering
	\vspace{-0.4em}
	\bigskip
	\begin{lrbox}{\tablebox}
		\begin{tabular}{cccccccccc}
			\hline
			& &\multicolumn{2}{c}{$\theta = 0.8$}&
			&\multicolumn{2}{c}{$\theta = 1.0$}&&\multicolumn{2}{c}{$\theta = 1.2$}\\
			\cline{3-4}   \cline{6-7} \cline{9-10}
			&         {$N$}      & $L^2$ error& Order &~~~& $L^2$ error & Order & ~~~ &  $L^2$  error  & Order \\
			\hline
			\multirow{4}{*}{$P^0$}
			&$10$&   1.06  &  --   &&   7.37E-01  &  --   &&  7.73E-01  &  --  \\
			&$20$&   4.49E-01  & 1.23  &&   3.68E-01  & 1.00  &&  3.87E-01  & 1.00\\
			&$40$&   2.07E-01  & 1.12  &&   1.83E-01  & 1.01  &&  1.90E-01  & 1.03 \\
			&$80$&   1.00E-01  & 1.05  &&   9.11E-02  & 1.01  &&  9.57E-02  & 0.99 \\ \vspace{-0.8em}
			\\
			\multirow{4}{*}{$P^1$}
			&$10$&   2.62E-01  &  --   &&   1.93E-01  &  --   &&  2.55E-01  &  --  \\
			&$20$&   9.51E-02  & 1.47  &&   4.92E-02  & 1.97  &&  6.00E-02  & 2.09\\
			&$40$&   2.83E-02  & 1.75  &&   1.24E-02  & 1.99  &&  1.42E-02  & 2.08 \\
			&$80$&   7.36E-03  & 1.94  &&   3.11E-03  & 2.00  &&  3.45E-03  & 2.04 \\ \vspace{-0.8em}
			\\
			\multirow{4}{*}{$P^2$}
			&$10$&   2.78E-03  &  --   &&   2.63E-03  &  --   &&  3.31E-03  &  --  \\
			&$20$&   3.51E-04  & 2.98  &&   3.32E-04  & 2.99  &&  4.32E-04  & 2.94\\
			&$40$&   4.53E-05  & 2.96  &&   4.16E-05  & 3.00  &&  5.84E-05  & 2.89 \\
			&$80$&   5.75E-06  & 2.98  &&   5.20E-06  & 3.00  &&  7.64E-06  & 2.93 \\ \vspace{-0.8em}
			\\
			\multirow{4}{*}{$P^3$}
			&$10$&   4.90E-04  &  --    &&  1.52E-04  &  --   &&  3.17E-04  &  --  \\
			&$20$&   3.63E-05  & 3.75  &&   9.61E-06  & 3.99  &&  1.89E-05  & 4.07\\
			&$40$&   7.34E-06  & 3.95  &&   6.02E-07  & 4.00  &&  1.15E-06  & 4.03 \\
			&$80$&   1.46E-07  & 4.00  &&   3.77E-08  & 4.00  &&  7.13E-08  & 4.02 \\ \vspace{-0.8em}
			\\
			\multirow{4}{*}{$P^4$}
			&$10$&   6.65E-06  &  --    &&  5.28E-06  &  --   &&  8.34E-06  &  --  \\
			&$20$&   2.22E-07  & 4.90  &&   1.66E-07  & 4.99  &&  2.64E-07  & 4.98\\
			&$40$&   7.45E-09  & 4.90  &&   5.21E-09  & 5.00  &&  9.43E-09  & 4.81 \\
			&$80$&   2.41E-10  & 4.95  &&   1.63E-10  & 5.00  &&  3.19E-10  & 4.88 \\ \vspace{-1.2em}
			\\
			\hline
		\end{tabular}
	\end{lrbox}
	\scalebox{1.04}{\usebox{\tablebox}}
\end{table}

\vspace{-0.3em}

\begin{table}[htp!]
	\caption{\label{table6} $L^2$ errors and  orders for Example \ref{ex4} with Dirichlet boundary conditions and different $\theta$, $T=1$.} \centering
	\vspace{-0.4em}
	\bigskip
	\begin{lrbox}{\tablebox}
		\begin{tabular}{cccccccccc}
			\hline
			& &\multicolumn{2}{c}{$\theta = 0.8$}&
			&\multicolumn{2}{c}{$\theta = 1.0$}&&\multicolumn{2}{c}{$\theta = 1.2$}\\
			\cline{3-4}   \cline{6-7} \cline{9-10}
			&         {$N$}      & $L^2$ error& Order &~~~& $L^2$ error & Order & ~~~ &  $L^2$  error  & Order \\
			\hline
			\multirow{4}{*}{$P^0$}
			&$10$&   4.24E-01  &  --   &&   4.12E-01  &  --   &&  4.41E-01  &  --  \\
			&$20$&   1.74E-01  & 1.29  &&   1.90E-01  & 1.12  &&  2.28E-01  & 0.95\\
			&$40$&   7.90E-02  & 1.14  &&   8.58E-02  & 1.14  &&  1.02E-01  & 1.16 \\
			&$80$&   3.80E-02  & 1.06  &&   4.06E-02  & 1.08  &&  4.66E-02  & 1.13 \\ \vspace{-0.8em}
			\\
			\multirow{4}{*}{$P^1$}
			&$10$&   7.60E-02  &  --   &&   7.76E-02  &  --   &&  8.16E-02  &  --  \\
			&$20$&   1.48E-02  & 2.36  &&   1.19E-02  & 2.70  &&  1.17E-02  & 2.80\\
			&$40$&   3.44E-03  & 2.11  &&   2.40E-03  & 2.31  &&  2.22E-03  & 2.40 \\
			&$80$&   8.50E-04  & 2.02  &&   5.62E-04  & 2.09  &&  5.01E-04  & 2.15 \\ \vspace{-0.8em}
			\\
			\multirow{4}{*}{$P^2$}
			&$10$&   3.24E-03  &  --   &&   3.21E-03  &  --   &&  3.28E-03  &  --  \\
			&$20$&   3.38E-04  & 3.26  &&   3.80E-04  & 3.08  &&  4.86E-04  & 2.76\\
			&$40$&   4.01E-05  & 3.07  &&   4.38E-05  & 3.12  &&  5.97E-05  & 3.03 \\
			&$80$&   4.99E-06  & 3.01  &&   5.28E-06  & 3.05  &&  7.27E-06  & 3.04 \\ \vspace{-0.8em}
			\\
			\multirow{4}{*}{$P^3$}
			&$10$&   4.33E-04  &  --    &&  4.67E-04  &  --   &&  1.93E-04  &  --  \\
			&$20$&   1.93E-05  & 4.49  &&   1.62E-05  & 4.85  &&  9.88E-06  & 4.29\\
			&$40$&   1.02E-06  & 4.23  &&   7.17E-07  & 4.50  &&  5.76E-07  & 4.10 \\
			&$80$&   6.14E-08  & 4.06  &&   3.95E-08  & 4.18  &&  3.51E-08  & 4.04 \\ \vspace{-0.8em}
			\\
			\multirow{4}{*}{$P^4$}
			&$10$&   8.66E-06  &  --    &&  8.21E-06  &  --   &&  6.75E-06  &  --  \\
			&$20$&   2.06E-07  & 5.40  &&   2.33E-07  & 5.14  &&  2.59E-07  & 4.70\\
			&$40$&   5.87E-09  & 5.13  &&   6.03E-09  & 5.27  &&  8.63E-09  & 4.91 \\
			&$80$&   1.82E-10  & 5.01  &&   1.71E-10  & 5.14  &&  2.74E-10  & 4.98 \\ \vspace{-1.2em}
			\\
			\hline
		\end{tabular}
	\end{lrbox}
	\scalebox{1.04}{\usebox{\tablebox}}
\end{table}

\begin{example}\label{ex5}
Consider the following problem
\begin{align*}
	\cut+\bs f(\bs u)_x&=(\bs A(\bs u)\bs u_x)_x+\bs g,\quad (x,t)\in(0,2\pi)\times(0,T]\\
	\bs u(x,0)&=\bs u_0(x),\quad\quad \qquad ~~~~x\in (0,2\pi),
\end{align*}
with nonlinear diffusive terms and periodic boundary conditions, where $\bs u=(\uo,\ut)^\trs$, $\bs f(\bs u)=(\uo+\ut,\uo+\ut)^\trs$, $\bs A(\bs u)=diag(\uo^4,\ut^4)$. The source term and initial condition are chosen such that the exact solution is
\begin{equation*}
	\bs u=(\sin(x-t),\sin(x-t))^\trs.
	\end{equation*}	
\end{example}
The $L^2$ errors and numerical orders for Example \ref{ex5} are given in Table \ref{table7}, from which we can see that optimal error estimates also hold for problems with nonlinear diffusive coefficients.

\begin{table}[htp!]
	\caption{\label{table7} $L^2$ errors and  orders for Example \ref{ex5} with nonlinear diffusive terms and different $\theta$, $T=0.5$.} \centering
	\vspace{-0.4em}
	\bigskip
	\begin{lrbox}{\tablebox}
		\begin{tabular}{cccccccccc}
			\hline
			& &\multicolumn{2}{c}{$\theta = 0.8$}&
			&\multicolumn{2}{c}{$\theta = 1.0$}&&\multicolumn{2}{c}{$\theta = 1.2$}\\
			\cline{3-4}   \cline{6-7} \cline{9-10}
			&         {$N$}      & $L^2$ error& Order &~~~& $L^2$ error & Order & ~~~ &  $L^2$  error  & Order \\
			\hline
			\multirow{4}{*}{$P^0$}
			&$15$&   4.15E-01  &  --   &&   5.26E-01  &  --   &&  6.54E-01  &  --  \\
			&$30$&   2.10E-01  & 0.98  &&   2.76E-01  & 0.93  &&  3.50E-01  & 0.90\\
			&$45$&   1.41E-01  & 0.99  &&   1.87E-01  & 0.96  &&  2.39E-01  & 0.94 \\
			&$60$&   1.06E-01  & 0.99  &&   1.42E-01  & 0.97  &&  1.82E-01  & 0.95 \\ \vspace{-0.8em}
			\\
			\multirow{4}{*}{$P^1$}
			&$15$&   3.74E-02  &  --   &&   2.68E-02  &  --   &&  2.26E-02  &  --  \\
			&$30$&   9.68E-03  & 1.95  &&   6.70E-03  & 2.00  &&  5.59E-03  & 2.01\\
			&$45$&   4.32E-03  & 1.99  &&   2.98E-03  & 2.00  &&  2.48E-03  & 2.00 \\
			&$60$&   2.43E-03  & 2.00  &&   1.67E-03  & 2.00  &&  1.39E-03  & 2.00 \\ \vspace{-0.8em}
			\\
			\multirow{4}{*}{$P^2$}
			&$15$&   7.89E-04  &  --   &&   9.71E-04  &  --   &&  1.16E-03  &  --  \\
			&$30$&   9.22E-05  & 3.10  &&   1.17E-04  & 3.06  &&  1.45E-04  & 3.01\\
			&$45$&   2.69E-05  & 3.04  &&   3.41E-05  & 3.04  &&  4.25E-05  & 3.02 \\
			&$60$&   1.13E-05  & 3.02  &&   1.43E-05  & 3.02  &&  1.79E-05  & 3.02 \\ \vspace{-0.8em}
			\\
			\multirow{4}{*}{$P^3$}
			&$10$&   1.92E-04  &  --   &&   1.60E-04  &  --   &&  1.47E-04  &  --  \\
			&$15$&   3.83E-05  & 3.97  &&   2.80E-05  & 4.29  &&  2.44E-05  & 4.43\\
			&$25$&   4.83E-06  & 4.05  &&   3.30E-06  & 4.19  &&  2.76E-06  & 4.27 \\
			&$30$&   2.29E-06  & 4.09  &&   1.56E-06  & 4.12  &&  1.30E-06  & 4.12 \\ \vspace{-0.8em}
			\\
			\multirow{4}{*}{$P^4$}
			&$10$&   1.81E-05  &  --   &&   1.22E-05  &  --   &&  1.02E-05  &  --  \\
			&$15$&   1.05E-06  & 7.03  &&   1.12E-06  & 5.88  &&  1.22E-06  & 5.22\\
			&$25$&   4.43E-08  & 6.19  &&   6.38E-08  & 5.62  &&  8.33E-08  & 5.26 \\
			&$30$&   1.54E-08  & 5.81  &&   2.32E-08  & 5.55  &&  3.14E-08  & 5.34 \\ \vspace{-1.2em}
			\\
			\hline
		\end{tabular}
	\end{lrbox}
	\scalebox{1.04}{\usebox{\tablebox}}
\end{table}

\begin{example}\label{ex6}
In this example, consider the following fully nonlinear degenerate convection-diffusion Buckley-Leverett system with discontinuous initial data
	\begin{align*}
	\cut+\bs f(\bs u)_x&=(\bs A(\bs u)\bs u_x)_x+\bs g,\quad (x,t)\in(0,1)\times(0,T]\\
	\bs u(x,0)&=\bs u_0(x),\quad\quad \qquad ~~~~ x\in (0,1),
	\end{align*}
and Dirichlet boundary conditions
	\begin{align*}
	\uo(0,t)&=1,\quad\uo(1,t)=0,\\
	\ut(0,t)&=1,\quad\ut(1,t)=0,
	\end{align*}
where $\bs u=(\uo,\ut)^\trs,$ \textcolor{blue}{$\bs f(\bs u)=\Big(\frac{\ut^2}{\ut^2+(1-\ut)^2}, \frac{\uo^2}{\uo^2+(1-\uo)^2}\Big)^\trs$}, $\bs A(\bs u)=diag(0.01a(\uo),0.01a(\ut))$. The coefficient $a(w)$ and initial condition $u_1(x,0),u_2(x,0)$ are
	$$a(w)=
	\begin{cases}
	4w(1-w),&\text{$0 \le w \le 1$},\\
	0,&\text{otherwise},
	\end{cases}
	\hspace{1em}
	u_1(x,0)= u_2(x,0) =
	\begin{cases}
	1-3x,&\text{$0 \le x \le 1/3$},\\
	0,&\text{$1/3\le x\le 1$}.
	\end{cases}
	$$
\end{example}
\vspace{-0.5em}
\begin{figure}[!htbp]
	\centering
	\subfigure[$P^1$ with $\theta=1.0$, $N=160$, $T=0.2$]{
		\begin{minipage}[t]{0.5\linewidth}
			\centering
			\includegraphics[width=\linewidth]{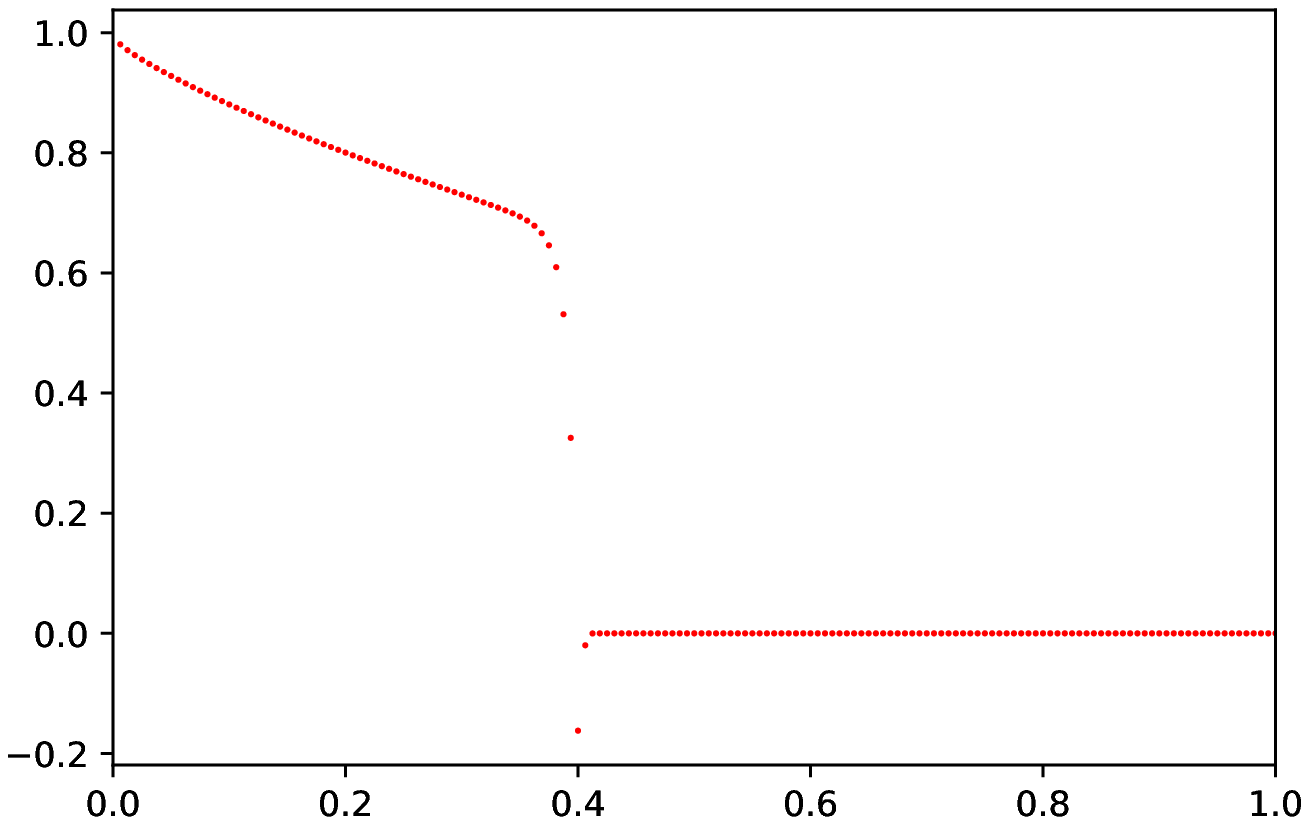}
			%\caption{fig1}
		\end{minipage}%
	}%
	\subfigure[$P^1$ with $\theta=1.3$, $N=160$, $T=0.2$]{
		\begin{minipage}[t]{0.5\linewidth}
			\centering
			\includegraphics[width=\linewidth]{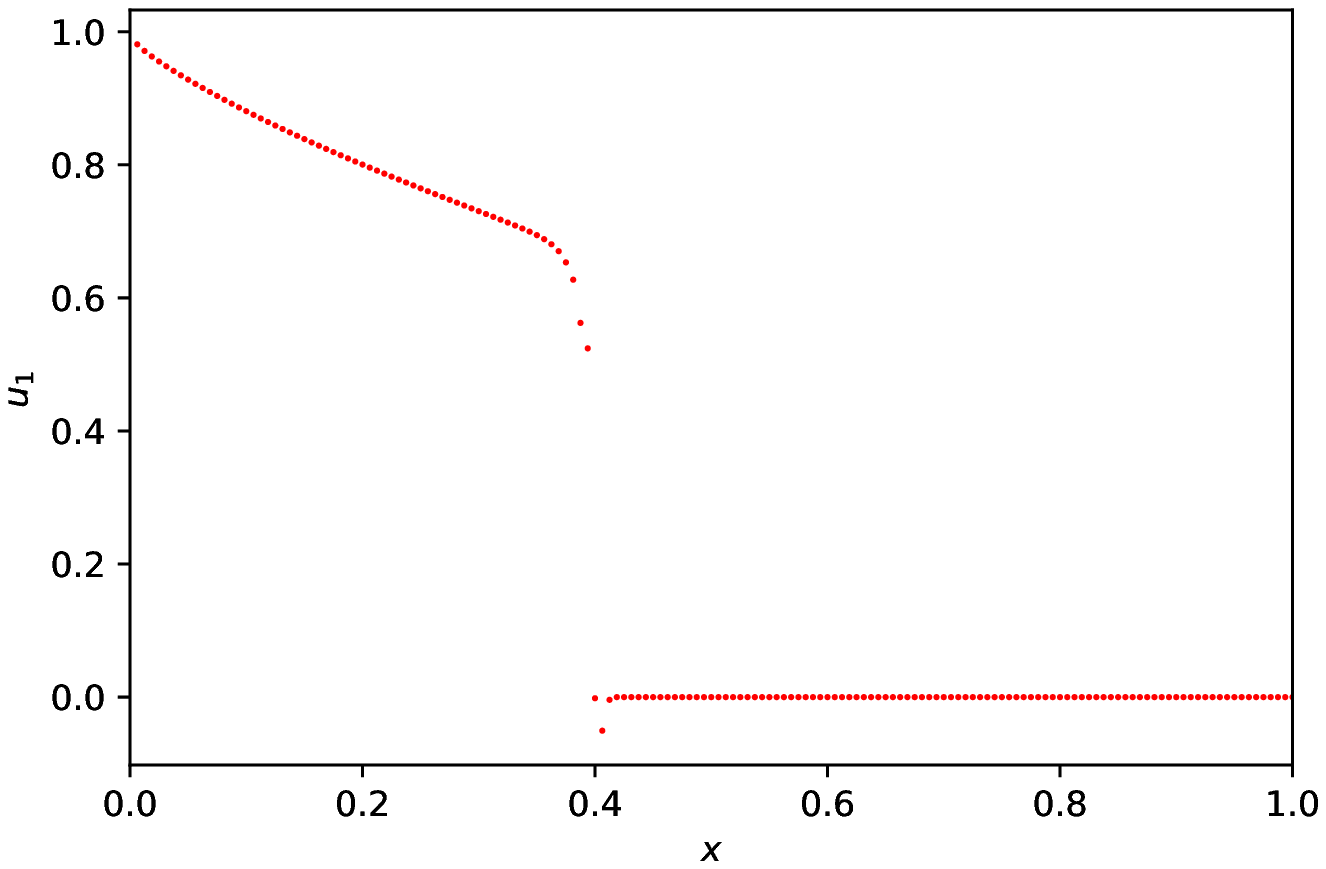}
			%\caption{fig2}
		\end{minipage}%
	}%
	
	\subfigure[$P^2$ with $\theta=1.0$, $N=160$, $T=0.1$]{
		\begin{minipage}[t]{0.5\linewidth}
			\centering
			\includegraphics[width=\linewidth]{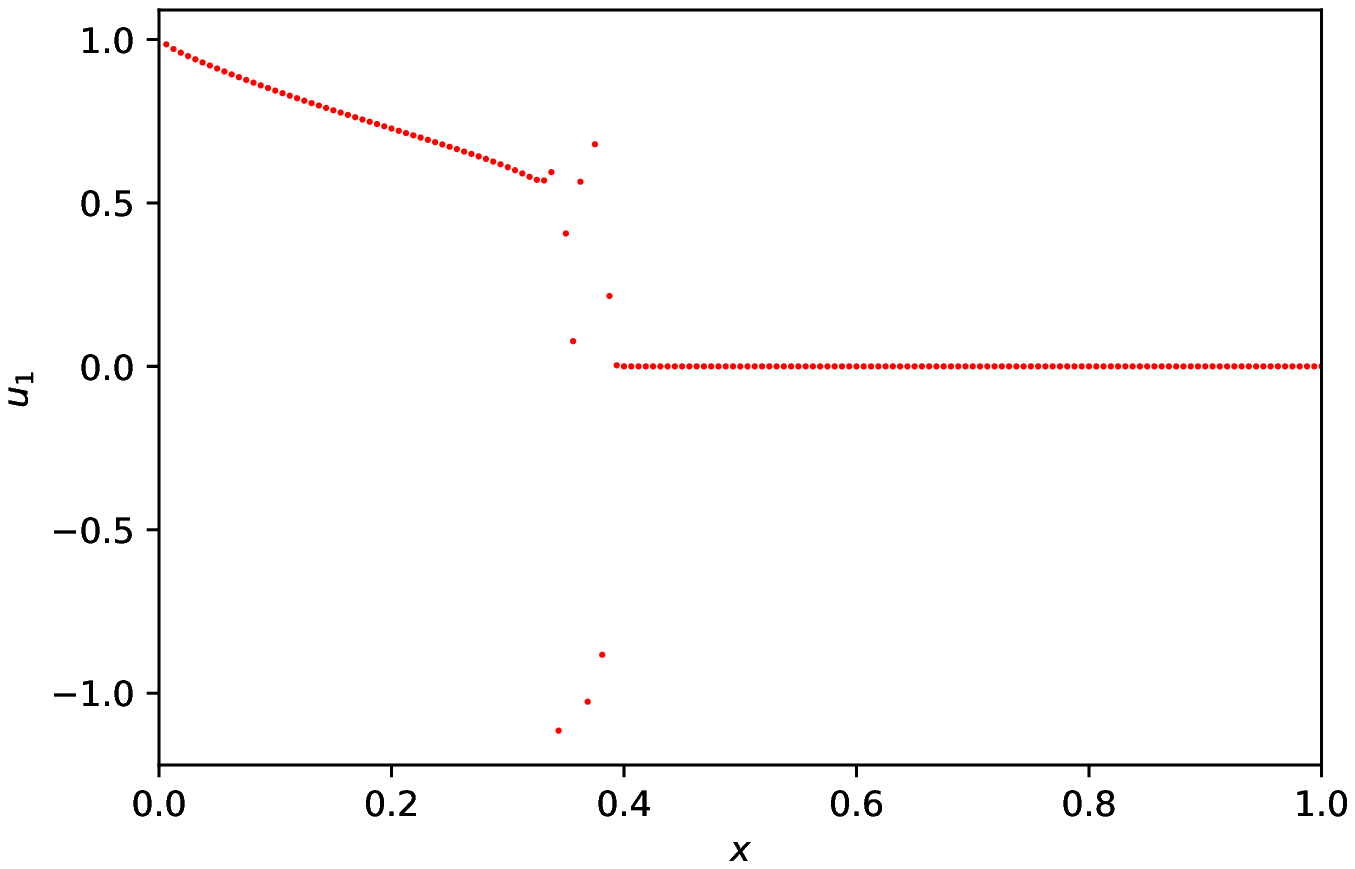}
			%\caption{fig2}
		\end{minipage}
	}%
	\subfigure[$P^2$ with $\theta=0.8$, $N=160$, $T=0.1$]{
		\begin{minipage}[t]{0.5\linewidth}
			\centering
			\includegraphics[width=\linewidth]{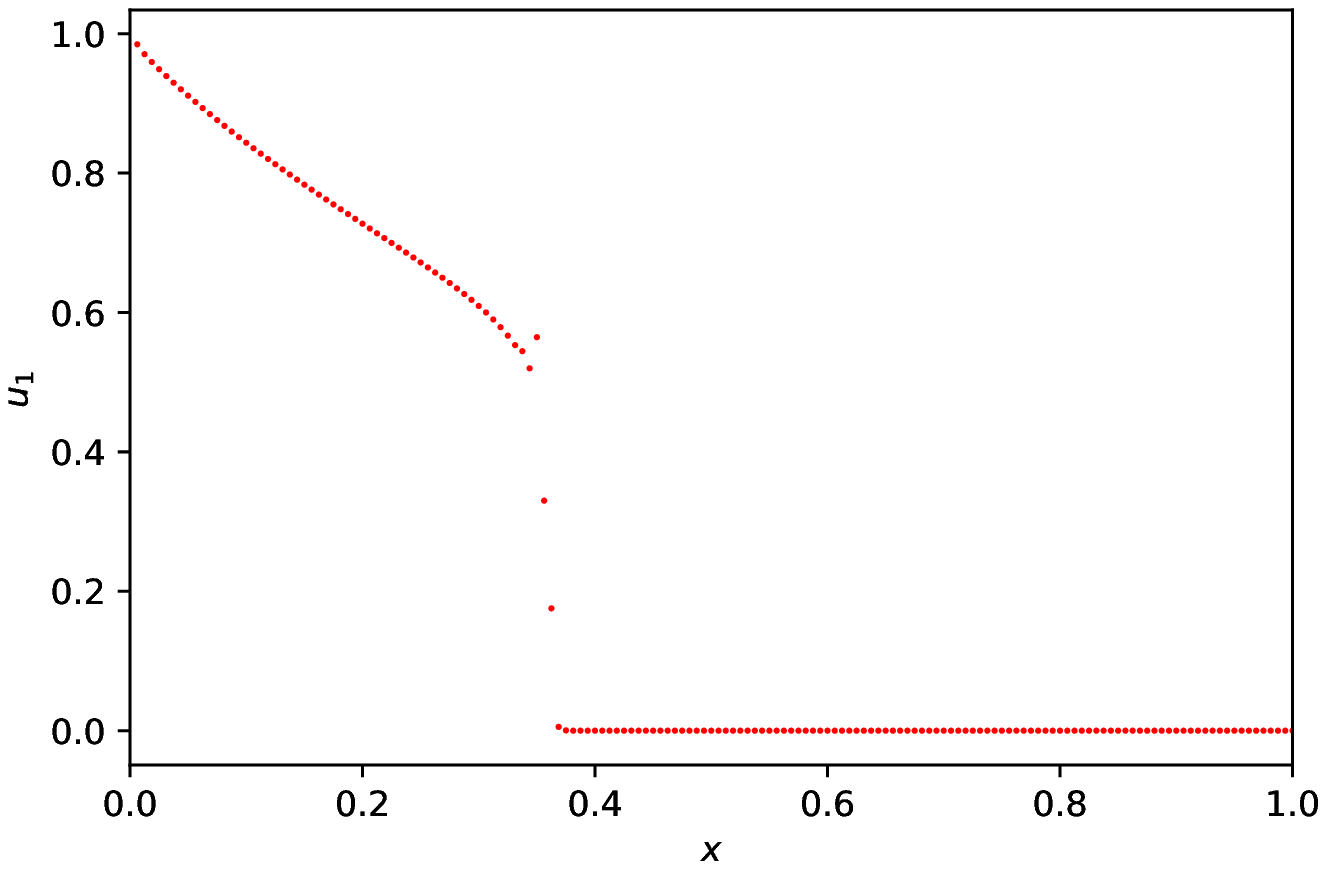}
			%\caption{fig2}
		\end{minipage}
	}%
	
	\centering
	\caption{LDG solutions to Example \ref{ex6} with a degenerate nonlinear diffusion term using $P^1, P^2$ polynomials}\label{figure2}
\end{figure}

The LDG solutions to Example \ref{ex6} are shown in Figure \ref{figure2}, in which $P^1, P^2$ polynomials and different weights $\theta$ are considered. From the figure, we can see that, for $P^1$ polynomials, the LDG solution with $\theta = 1.3$ is better than that for the traditional upwind flux with $\theta = 1$ and that, for $P^2$ polynomials, the LDG solution with $\theta = 0.8$ is better than that for the traditional upwind flux with $\theta = 1$, as far as the discontinuity transitions are concerned. This illustrates that it is beneficial to use the LDG scheme with generalized numerical fluxes solving fully nonlinear degenerate convection-diffusion systems, producing a satisfactory approximation even without the aid of any limiter.
Besides, it seems that for even (odd) values of $k$, smaller (larger) $\theta$ would lead to better approximations, which agrees with the case of smooth solutions when the magnitude of errors is considered, e.g. \cite{CMZ2017}.

\section{Concluding Remarks}\label{sc5}
In this paper, we carry out optimal error analysis of the LDG method using generalized numerical fluxes for  nonlinear convection-diffusion systems. For both \textcolor{blue}{symmetric positive definite} and symmetrizable flux Jacobian matrices, we derive optimal error estimates by analyzing suitable GGR projections and using the local characteristic decomposition of the flux Jacobian. A series of numerical experiments are given to validate the theoretical results. In future work, we will concentrate on the LDG method for multi-dimensional nonlinear convection-diffusion systems.

\Acknowledgements{The authors thank referees for their valuable suggestions that result in the improvement of the paper. This work was supported by the National Natural Science Foundation of China (Grant No. 11971132, 11971131) and the Natural Science Foundation of Heilongjiang Province, China (Grant No. YQ2021A002). Additional support was provided by Guangdong Basic and Applied Basic Research Foundation (Grant No. 2020B1515310006)}



\begin{thebibliography}{99}
\bahao\baselineskip 11.5pt

\bibitem{BS1994}
Brenner S C, Scott L R. The mathematical theory of finite element methods. Texts in Applied Mathematics,
volume 15 Springer-Verlag, New York, 1994

\bibitem{BX2018}
Buli J, Xing Y. Local discontinuous Galerkin methods for the Boussinesq coupled BBM system. J Sci Comput, 2018,
75: 536--559

\bibitem{Castillo2020}
Castillo P, G\'{o}mez S. On the convergence of the local discontinuous Galerkin method applied to a stationary one dimensional fractional diffusion problem. J Sci Comput, 2020,
85: article number 32

\bibitem{Cheng2019}
Cheng Y. Optimal error estimate of the local discontinuous Galerkin methods based on the generalized alternating
numerical fluxes for nonlinear convection-diffusion equations. Numer Algorithms, 2019, 80: 1329--1359

\bibitem{Cheng2021}
Cheng Y. On the local discontinuous Galerkin method for singularly perturbed problem with two parameters.
J Comput Appl Math, 2021, 392: 113485

\bibitem{CMZ2017}
Cheng Y, Meng X, Zhang Q. Application of generalized Gauss-Radau projections for the local discontinuous
Galerkin method for linear convection-diffusion equations. Math Comput, 2017, 86: 1233--1267

\bibitem{CS1990}
Cockburn B, Hou S, Shu C W. The Runge-Kutta local projection discontinuous Galerkin finite element method
for conservation laws. IV. The multidimensional case. Math Comput, 1990, 54: 545--581

\bibitem{CS19893}
Cockburn B, Lin S Y, Shu C W. TVB Runge-Kutta local projection discontinuous Galerkin finite element method
for conservation laws. III. One-dimensional systems.
J Comput Phys, 1989, 84: 90--113

\bibitem{CS1989}
Cockburn B, Shu C W. TVB Runge-Kutta local projection discontinuous Galerkin finite element method for conservation laws. II. General framework.
Math Comput, 1989, 52: 411--435

\bibitem{CS19982}
Cockburn B, Shu C W. The local discontinuous Galerkin method for time-dependent convection-diffusion systems.
SIAM J Numer Anal, 1998, 35: 2440--2463

\bibitem{DPJ1979}
Davis P J. Circulant matrices.
John Wiley \& Sons, New York-Chichester-Brisbane, 1979,
A Wiley-Interscience Publication, Pure and Applied Mathematics

\bibitem{GX2021}
Guo R, Xing Y. Optimal energy conserving local discontinuous Galerkin methods for elastodynamics: semi and
fully discrete error analysis.
J Sci Comput, 2021, 87: 13

\bibitem{Li2020}
Li J, Zhang D, Meng X, Wu B, Zhang Q. Discontinuous Galerkin methods for nonlinear scalar conservation laws: generalized local Lax-Friedrichs numerical fluxes. SIAM J Numer Anal, 2020, 58: 1--20

\bibitem{Li2020-2}
Li J, Zhang D, Meng X, Wu B. Analysis of local discontinuous Galerkin methods with generalized numerical fluxes for linearized KdV equations. Math Comput, 2020, 89: 2085--2111

\bibitem{Liu15}
Liu H, Ploymaklam N. A local discontinuous Galerkin method for the Burgers-Poisson equation. Numer Math, 2015,
129: 321--351

%\bibitem{L2021}
%Liu X, Zhang D, Meng X, Wu B. Superconvergence of local discontinuous Galerkin methods with generalized alternating fluxes for 1D linear convection-diffusion equations. Sci China Math, 2021, 64: 1305--1320

%\bibitem{LZMW2021}
%Liu X, Zhang D, Meng X, Wu B. Superconvergence of the local discontinuous Galerkin method for one dimensional nonlinear convection-diffusion equations.
%J Sci Comput, 2021, 87: article number 39

\bibitem{LE2011}
Lott P A, Elman H. Fast iterative solver for convection-diffusion systems with spectral elements. Numer Methods Partial Differ Equ, 2011, 27: 231--254

\bibitem{Luo15}
Luo J, Shu C W, Zhang Q. A priori error estimates to smooth solutions of the third order Runge-Kutta discontinuous Galerkin method for symmetrizable systems of conservation laws. ESAIM Math Model Numer Anal, 2015, 49: 991--1018

\bibitem{MS2017}
May S. Spacetime discontinuous Galerkin methods for solving convection-diffusion systems. ESAIM Math Model
Numer Anal, 2017, 51: 1755--1781

\bibitem{MN2016}
Mazaheri A, Nishikawa H. Efficient high-order discontinuous Galerkin schemes with first-order hyperbolic advection-diffusion system approach.
J Comput Phys, 2016, 321: 729--754

\bibitem{MRJ2018}
Meng X, Ryan J K. Divided difference estimates and accuracy enhancement of discontinuous Galerkin methods
for nonlinear symmetric systems of hyperbolic conservation laws. IMA J Numer Anal, 2018, 38: 125--155

\bibitem{Meng2016}
Meng X, Shu C W, Wu B. Optimal error estimates for discontinuous Galerkin methods based on upwind-biased
fluxes for linear hyperbolic equations. Math Comput, 2016, 85: 1225--1261

\bibitem{MAP2017}
Michoski C, Alexanderian A, Paillet C, et al. Stability of nonlinear convection-diffusion-reaction systems in discontinuous Galerkin methods.
J Sci Comput, 2017, 70: 516--550

\bibitem{RBM2012}
Rossi F, Budroni M A, Marchettini N, et al. Segmented waves in a reaction-diffusion-convection system. Chaos, 2012, 22: 037109

\bibitem{TX2019}
Tao Q, Xia Y. Error estimates and post-processing of local discontinuous Galerkin method for Schr\"odinger equations.
J Comput Appl Math, 2019, 356: 198--218

\bibitem{TXK2016}
Tian L, Xu Y, Kuerten J, et al. An h-adaptive local discontinuous Galerkin method for the Navier-Stokes-Korteweg equations. J Comput Phys, 2016, 319: 242--265

\bibitem{wang2020}
Wang H, Zhang Q, Wang S, Shu C W. Local discontinuous Galerkin methods with explicit-implicit-null time discretizations for solving nonlinear diffusion problems. Sci China Math, 2020, 63: 183--204

\bibitem{xs2007}
Xu Y, Shu C W. Error estimates of the semi-discrete local discontinuous Galerkin method for nonlinear convection-diffusion and KdV equations. Comput Methods Appl Mech Eng, 2007, 196: 3805--3822

\bibitem{xu2005local}
Xu Y, Shu C W. Local discontinuous galerkin methods for nonlinear Schr\"odinger equations. J Comput Phys, 2005, 205:
72--97


\bibitem{YS2002}
Yan J, Shu C W. A local discontinuous Galerkin method for KdV type equations. SIAM J Numer Anal, 2002, 40: 769--791

\bibitem{ZS2006}
Zhang Q, Shu C W. Error estimates to smooth solutions of Runge-Kutta discontinuous Galerkin method for symmetrizable systems of conservation laws. SIAM J Numer Anal, 2006, 44: 1703--1720

\bibitem{ZG1995}
Zhou G. A local $L^2$-error analysis of the streamline diffusion method for nonstationary convection-diffusion systems. RAIRO Mod\'{e}l Math Anal Num\'{e}r, 1995, 29: 577--603

\end{thebibliography}
\end{document}